\documentclass[11pt,a4paper,reqno]{amsart}
\usepackage[myheadings]{fullpage}
\usepackage{enumerate}
\usepackage{mathtools}
\usepackage{amssymb}
\usepackage{float}
\usepackage{hyperref}
\hypersetup{colorlinks=true, linkcolor=blue,citecolor=blue}

\newtheorem{cor*}{Corollary}
\newtheorem{prop*}{Proposition}
\newtheorem{thm*}{Theorem}

\newtheorem{theorem}{Theorem}[section]
\newtheorem{cor}{Corollary}[theorem]
\newtheorem{prop}[theorem]{Proposition}
\newtheorem{lem}[theorem]{Lemma}
\newtheorem{algo}[theorem]{Algorithm}

\theoremstyle{definition}

\newtheorem{rem}[theorem]{Remark}

\newcommand{\f}{\mathfrak{f}}
\newcommand{\h}{\mathfrak{h}}
\newcommand{\Z}{\mathbb{Z}}
\newcommand{\s}{\mathbb{S}}
\newcommand{\orb}{\mathcal{O}}
\newcommand{\N}{\mathcal{N}}
\newcommand{\I}{\mathcal{I}}
\newcommand{\wt}{\widetilde}

\DeclareMathOperator{\aut}{Aut}
\DeclareMathOperator{\stab}{Stab}
\DeclareMathOperator{\lmap}{LMod}
\DeclareMathOperator{\smap}{SMod}
\DeclareMathOperator{\map}{Mod}
\DeclareMathOperator{\symp}{Sp}
\DeclareMathOperator{\Sl}{SL}
\DeclareMathOperator{\homeo}{Homeo^{+}}
\everymath{\displaystyle}

\begin{document}

\title[Liftable mapping class groups of regular abelian covers]{Liftable mapping class groups of\\ regular abelian covers}

\author{Neeraj K. Dhanwani}
\address{Department of Mathematical Sciences\\
Indian Institute of Science Education and Research Mohali\\
Knowledge city, Sector 81, SAS Nagar, Manauli PO 140306\\
Punjab, India}
\email{neerajk.dhanwani@gmail.com}

\author{Pankaj Kapari}
\address{Department of Mathematical Sciences\\
Indian Institute of Science Education and Research Mohali\\
Knowledge city, Sector 81, SAS Nagar, Manauli PO 140306\\
Punjab, India}
\email{pankajkapri02@gmail.com}
\urladdr{https://sites.google.com/view/pankajkapdi/home}

\author{Kashyap Rajeevsarathy}
\address{Department of Mathematics\\
Indian Institute of Science Education and Research Bhopal\\
Bhopal Bypass Road, Bhauri \\
Bhopal 462066, Madhya Pradesh\\
India}
\email{kashyap@iiserb.ac.in}
\urladdr{https://home.iiserb.ac.in/$_{\widetilde{\phantom{n}}}$kashyap/}

\author{Ravi Tomar}
\address{Department of Mathematical Sciences\\
Indian Institute of Science Education and Research
Berhampur\\
Laudigam, Village Road, Brahmapur\\
Odisha 760003, India}
\email{ravitomar547@gmail.com}



\begin{abstract}
Let $S_g$ be the closed oriented surface of genus $g \geq 0$, and let $\mathrm{Mod}(S_g)$ be the mapping class group of $S_g$. For $g\geq 2$, we develop an algorithm to obtain a finite generating set for the liftable mapping class group $\mathrm{LMod}_p(S_g)$ of a regular abelian cover $p$ of $S_g$. A key ingredient of our method is a result that provides a generating set of a group $G$ acting on a connected graph $X$ such that the quotient graph $X/G$ is finite. As an application of our algorithm, when $k$ is prime, we provide a finite generating set for $\mathrm{LMod}_{p_k}(S_2)$ for cyclic cover $p_k:S_{k+1}\to S_2$. Using the Birman-Hilden theory, when $k=2,3$ and $g=2$, we also obtain a finite generating set for the normalizer of the Deck transformation group of $p_k$ in $\mathrm{Mod}(S_{k+1})$. We conclude the paper with an application of our algorithm that gives a finite generating set for $\mathrm{LMod}_p(S_2)$, where $p:S_5\to S_2$ is a cover with deck transformation group isomorphic to $\mathbb{Z}_2\oplus \mathbb{Z}_2$.
\end{abstract}

\maketitle
\section{Introduction}
\label{sec:into}
For $g,n\geq 0$, let $S=S_{g,n}$ be the connected, closed, and oriented surface of genus $g$ with $n$ marked points (or punctures). (We will denote $S_{g,0}$ simply by $S_g$.) Let $\homeo(S)$ be the group of orientation-preserving self-homeomorphisms of $S$ that preserve the set of marked points equipped with the compact-open topology. The \textit{mapping class group} of $S$ is the group $\map(S):=\pi_0(\homeo(S))$, and the elements of $\map(S)$ are called \textit{mapping classes}. In this paper, we develop an algorithm (see Algorithm~\ref{algo}) to obtain a finite generating set for the liftable mapping class group of a regular free abelian cover of $S_g$ for $g \geq 2$. A key ingredient in our method is a result (see Theorem~\ref{thm:genset_gpaction_graph}) that provides a generating set of a group $G$ that acts on a connected graph $X$ having finitely many vertex and edge orbits. Roughly speaking, our result asserts that if $X/G$ is a finite graph, then $G$ is generated by a collection of stabilizer subgroups of $G$ of vertices in $X$ that is in correspondence with the vertices of $X/G$ and a carefully chosen set of elements of $G$ in correspondence with the edges of $X/G$. 

For $g\geq 2$, as a consequence of the Nielsen-Kerckhoff theorem~\cite[Theorem 5]{kerckhoff83}, a finite subgroup $H < \map(S_g)$ acts on $S_g$ via a group $\tilde{H}$ of isometries of some hyperbolic metric on $S_g$ such that $H \cong \tilde{H}$. The $\tilde{H}$-action on $S_g$ induces a finite-sheeted branched cover $p:S_g\to S_g/\tilde{H}$. The orbit space $S_g/\tilde{H}$ is a hyperbolic orbifold associated with $\tilde{H}$-action on $S_g$ and will be denoted by $\orb_H$. We have $\orb_H\approx S_{g_0,n}$ for some integers $g_0,n\geq 0$ which are genus and number of branched points of $\orb_H$, respectively. We define $\map(\orb_H):=\map(S_{g_0,n})$. Let $\lmap_p(\orb_H)$, known as the \textit{liftable mapping class group of $p$}, be the subgroup of $\map(\orb_H)$ consisting of mapping classes represented by homeomorphisms that lift under the cover $p$. We define $\smap_p(S_g)$, known as \textit{symmetric mapping class group of $p$}, as the subgroup of $\map(S_g)$ consisting of mapping classes represented by homeomorphisms that preserve the fibers of $p$.   

The covering $p:S_g\to S_g/\tilde{H}$ is said to satisfy the \textit{Birman-Hilden property} if two isotopic fiber-preserving homeomorphisms are isotopic through fiber-preserving homeomorphisms. By the work of Birman-Hilden~\cite[Theorem 1]{birman73} and Harvey-Maclachlan~\cite[Corollary 12]{harvey75}, it follows that any finite-sheeted regular branched cover of a hyperbolic surface has the Birman-Hilden property. It was shown in~\cite{birman73, harvey75} that $p:S_g\to S_g/\tilde{H}$ satisfies the Birman-Hilden property if and only if the following short-exact sequence exists:
\begin{equation}
\label{eqn:bh_short_exact}
1 \longrightarrow H \longrightarrow \smap_p(S_g) \xrightarrow{\varphi} \lmap_p(\orb_H) \longrightarrow 1.
\end{equation}
Furthermore, it was shown in~\cite[Theorem 4-5]{birman73} and~\cite[Theorem 10]{harvey75} that $\smap_p(S_g)$ is the normalizer $N_{\map(S_g)}(H)$ of $H$ in $\map(S_g)$. (We refer the interested reader to the survey article~\cite{margalit21} by Margalit-Winarski.) When $H$ is abelian, using the action of $\map(\orb_H)$ on the first integral homology group $H_1(\orb_H^{\circ};\Z)$, it can be seen that $\lmap_p(\orb_H)$ has finite index in $\map(\orb_H)$, where $\orb_H^{\circ}$ is the surface obtained by deleting the branched points from $\orb_H$. Since $\map(\orb_H)$ is finitely presented~\cite[Theorem 5.7]{primer}, $\lmap_p(\orb_H)$ and $N_{\map(S_g)}(H)$ are also finitely presented. Thus, given a finite abelian subgroup $H<\map(S_g)$, one way to derive a finite presentation of $N_{\map(S_g)}(H)$ is by lifting a finite presentation for $\lmap_p(\orb_H)$ under $\varphi$. 

By considering the cover $p :S_2 \to S_{0,6}$ induced by the hyperelliptic involution $\iota$ in $S_2$, Birman-Hilden~\cite[Theorem 8]{birman71} obtained the first known finite presentation of the centralizer $\map(S_2) = C_{\map(S_2)}(\iota)$ of $\iota$ by lifting a known presentation for $\map(S_{0,6}) = \lmap_p(S_{0,6})$. Recently, Ghaswala-Winarski classified~\cite[Theorem 1.3]{winarski17} cyclic covers $p:S_g\to S_{0,n}$ for which $\lmap_p(S_{0,n})= \map(S_{0,n})$. In another work, they obtained~\cite[Theorem 5.7]{ghaswala17} a finite presentation of the $\lmap_p(S_{0,n})$ associated with cover $p:S_g\to S_{0,n}$ induced by a balanced superelliptic map $F\in \map(S_g)$ for which $\lmap(S_{0,n}) \neq \map(S_{0,n})$. Subsequently, Hirose-Omori~\cite[Theorem 5.3]{hirose22} obtained a finite presentation of $N_{\map(S_g)}(F)$ by lifting the presentation of $\lmap_p(S_{0,n})$ under $\varphi$. Using an algebraic characterization of liftability due to Broughton~\cite[Theorem 3.2]{broughton92}, more recently, Kapari-Rajeevsarathy-Sanghi gave a method~\cite[Algorithm 4.3]{kapdi25} to write a finite generating set of $\lmap_p(S_{0,n})$ associated with cyclic covers $p:S_g\to S_{0,n}$. As an application, they obtained~\cite[Corollary 4.6.1]{kapdi25} a finite presentation for the normalizer and the centralizer of a reducible periodic mapping class of the highest order, which is $2g+2$ (due to Kasahara~\cite[Theorem 4.1]{kasahara}).

In the rest of the paper, we will only consider unbranched covers. Consider the $k$-sheeted cyclic cover $p_k:S_{g_k}\to S_g$ (see Figure \ref{fig:free_cover} for $k=3$ and $g=2$) induced by a free cyclic $\Z_k$-action on $S_{g_k}$, where $g_k=k(g-1)+1$.
\begin{figure}[ht]
\centering
\includegraphics[scale=1.2]{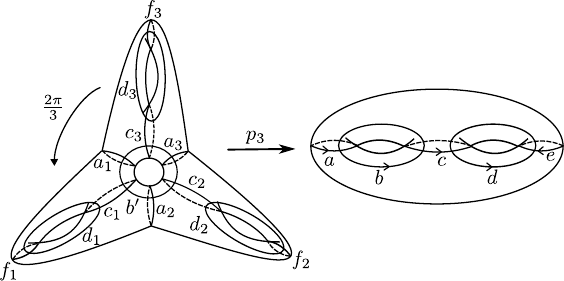}
\caption{A $3$-sheeted cover $S_4\to S_2$ induced by a free $2\pi/3$-rotation of $S_4$.}
\label{fig:free_cover}
\end{figure}
We denote $\lmap_{p_k}(S_g)$ by $\lmap_k(S_g)$. The action of $\map(S_g)$ on $H_1(S_g;\Z)$ induces the following short-exact sequence~\cite[Section 6.3]{primer}
\[
1\longrightarrow \I(S_g)\longrightarrow \map(S_g)\xlongrightarrow{\Psi} \symp(2g,\Z) \longrightarrow 1,
\]
where $\Psi$ is the \textit{symplectic representation} of $\map(S_g)$ and $\ker \Psi$ is the \textit{Torelli group} denoted by $\I(S_g)$. Using the symplectic characterization of $\lmap_k(S_g)$ due to Agarwal-Dey-Dhanwani-Rajeevsarathy~\cite[Theorem 2.2]{dhanwani22}, a finite generating set for $\Psi(\lmap_k(S_g))$ was obtained~\cite[Theorem 3.1]{dhanwani21}. For $g\geq 3$, due the work of Johnson~\cite{johnson83}, $\I(S_g)$ is finitely generated. For $g\geq 3$, Dey-Dhanwani-Patil-Rajeevsarathy~\cite[Theorem 3.9]{dhanwani21} derived a finite generating set for $\lmap_k(S_g)$ by combining generating sets of $\I(S_g)$ and that of $\Psi(\lmap_k(S_g))$. Due to the work of McCullough-Miller~\cite{mccullough86}, since $\mathcal{I}(S_2)$ is not finitely generated, this method does not work in the case of $g=2$.

Let $T_c$ denote the left-handed Dehn twist about a simple closed curve $c$ in $S$. Let $\Gamma_0(k)$ be the subgroup of $\Sl(2,\Z)$ defined as follows:
\[
\Gamma_0(k)=\left\{
\begin{pmatrix}
a & b \\ 
c & d
\end{pmatrix}\in \Sl(2,\Z)
\Big|~ c\equiv 0\pmod k 
\right\}.
\]
Let $\phi:\Sl(2,\Z)\to \symp(4,\Z)$ be the injective homomorphism defined as
\[
\phi:
\begin{pmatrix}
a & b \\ 
c & d
\end{pmatrix}
\mapsto
\begin{pmatrix}
a & b & 0 & 0 \\ 
c & d & 0 & 0 \\ 
0 & 0 & 1 & 0 \\ 
0 & 0 & 0 & 1
\end{pmatrix}.
\] 
Let $\s\subset\lmap_k(S_2)$ such that $\Psi(\s)$ is a finite generating set for $\phi(\Gamma_0(k))$. By applying our algorithm (Algorithm~\ref{algo}), we derive a finite generating set for $\lmap_k(S_2)$, where $k$ is prime (see Theorem~\ref{thm:genset_lmod}).

\begin{thm*}
\label{thm1}
For a prime number $k$ and the $k$-sheeted regular cyclic cover $p_k:S_{k+1}\to S_2$, we have $\lmap_k(S_2)=\langle \s\cup \{T_a,T_b^k,T_c,T_d,T_e,\iota\}\cup \s'\cup \s'' \rangle$, where $\s'=\{T_b^{1-j}T_aT_b^{1-\bar{j}}\mid 1\leq j<k,j\bar{j}\equiv 1\pmod k\}$ and $\s''=\{(T_bT_c)^6,T_b^iT_c^jT_d(T_bT_c)^6T_d^{-1}T_c^{-j}T_b^{-i}\mid 1\leq i,j<k\}$.
\end{thm*}

As a corollary of Theorem~\ref{thm1}, we recover (see Corollary~\ref{cor:genset_lmod2}) the finite generating set for $\lmap_2(S_2)$ obtained in~\cite[Corollary 5.5]{dhanwani21} and also derive an explicit finite generating set for $\lmap_3(S_2)$ (see Corollary~\ref{cor:genset_lmod3}).

\begin{cor*}
\label{cor1}
We have $\lmap_2(S_2)=\langle T_a,T_b^2,T_c,T_d,T_e\rangle$ and $\lmap_3(S_2)=\langle T_a,T_b^3,T_c,T_d,T_e,\iota\rangle$.
\end{cor*}

Let $F_{g,k}\in \map(S_{g_k})$ be a periodic mapping class of order $k$ inducing the cover $p_k$. For $g\geq 3$, an explicit finite generating set for $N_{\map(S_{g_2})}(F_{g,2})$ was obtained in~\cite[Corollary 3.10]{dhanwani21} by lifting the finite generating set of $\lmap_2(S_g)$ under $\varphi:N_{\map(S_{g_2})}(F_{g,2})\to \lmap_2(S_g)$. By lifting the generating set described in Corollary~\ref{cor1}, we have also obtained a similar result (see Corollary~\ref{cor:genset_norm23}).

\begin{figure}[ht]
\centering
\includegraphics[scale=0.4]{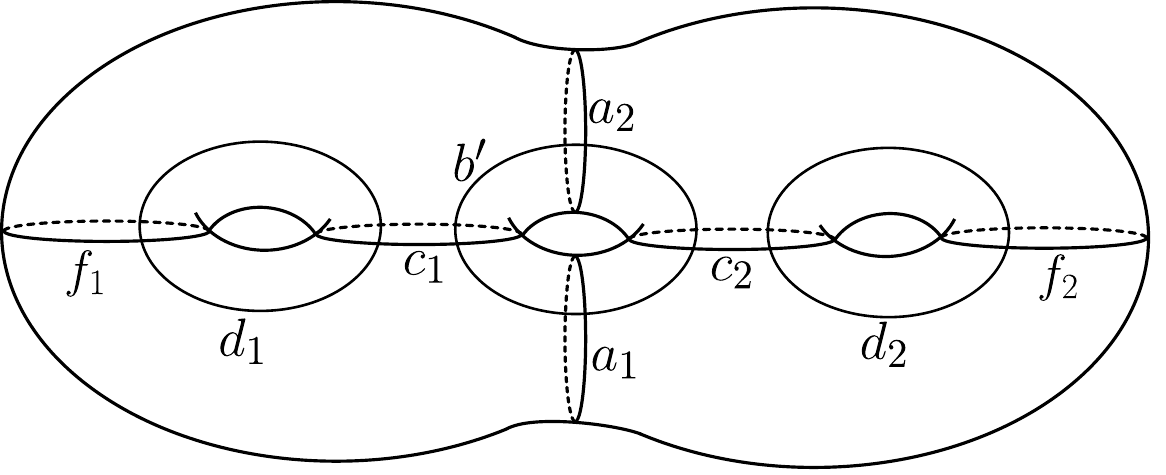}
\caption{The simple closed curves on $S_3$ used to describe a generating set for $N_{\map(S_3)}(F_{2,2})$.}
\label{s3_curves}
\end{figure} 

\begin{cor*}
\label{cor2}
We have
$$N_{\map(S_3)}(F_{2,2})=\langle F_{2,2}, T_{a_1}T_{a_2}, T_{b'}, T_{c_1}T_{c_2}, T_{d_1}T_{d_2},T_{f_1}T_{f_2}\rangle$$
and
$$N_{\map(S_4)}(F_{2,3})=\langle F_{2,3}, T_{a_1}T_{a_2}T_{a_3}, T_{b'}, T_{c_1}T_{c_2}T_{c_3}, T_{d_1}T_{d_2}T_{d_3},T_{f_1}T_{f_2}T_{f_3},\iota_1\iota_2\iota_3\rangle$$
for simple closed curves shown in Figure~\ref{fig:free_cover}-\ref{s3_curves} and $\iota_1,\iota_2,\iota_3\in \map(S_4)$ are involutions represented by $\pi$-rotation of $S_4$ as shown in Figure~\ref{fig:free_cover}. 
\end{cor*}

We conclude this paper by deriving a generating set of $\lmap_p(S_2)$ of an abelian cover \linebreak $p:S_5\to S_2$ with Deck group $\Z_2\oplus \Z_2=\langle s\mid s^2=1\rangle \oplus \langle t\mid t^2=1\rangle$ associated with the map $H_1(S_2;\Z)\cong \Z^4\to \Z_2\oplus \Z_2$ given by $(e_1,e_2,e_3,e_4)\mapsto (s,t,0,0)$ (see Theorem~\ref{thm:genset_klein_lmod}).

\begin{thm*}
\label{thm2}
For the cover $p$ described as above, we have $\lmap_p(S_2)=\langle T_a,T_b,T_c^2,T_d,T_e\rangle$.
\end{thm*}

\section{Main algorithm}
\label{sec:prelim}
\subsection{Generating set of a group acting on a graph}
\label{subsec:genset_gpaction_graph}
Let $G$ be a group that acts on a connected graph $X$ such that the sets of orbits of vertices and edges are finite. Let $Y$ denote the quotient graph of $X$ under the action of $G$. Since $X$ is connected, $Y$ is a connected finite graph.

\begin{theorem}
\label{thm:genset_gpaction_graph}
Let $T$ be a maximal tree in $Y$ and choose a lift $\widetilde{T}$ of $T$ in $X$. Let $v_1,v_2,\dots,v_n$ be the vertices of $T$ and $\wt{v}_1,\wt{v}_2,\dots,\wt{v}_n$ be the lifts of $v_i$'s in $\wt{T}$, respectively. Assume that the following conditions hold:
\begin{enumerate}[$(i)$]
\item Suppose, for some $1\leq i\leq n$, there exists $k_i\geq 1$ such that there are $k_i$-loops at the vertex $v_i$, and each of these loops has a lift in $X$ at $\wt{v}_i$ that is not a loop. For $1\leq j\leq k_i$, let $\{\xi_j^i\}$ be the set of loops at the vertex $v_i$. Denote by $\wt{\xi}_j^i$ a lift of $\xi_j^i$ at the vertex $\wt{v_i}$ that is not a loop. For each $i$ and $1\leq j\leq k_i$, suppose there exists $\f_j^i\in G$ such that $\f_j^i\wt{v}_i$ is the other vertex of $\wt{\xi}_j^i$ (if there is no such loop at $v_i$ for some $i$, then there is no $\f_j^i$).
			
\item Suppose, for some $l\geq 1$, there exist edges $\xi_1,\xi_2,\dots,\xi_l$ of $Y$ which are neither loops nor part of $T$. For $1\leq m\leq l$, let $\wt{\xi}_m$ be a lift of $\xi_m$ in $X$, where exactly one vertex of  each $\wt{\xi}_m$ belongs to the set of vertices of $\wt{T}$. Suppose the set of vertices of $\xi_m$ is $\{v_i,v_j\}$, and the set of vertices of $\wt{\xi_m}$ is $\{\wt{v_i},\wt{w}\}$. Then, there exists $\f_m\in G$ such that $\f_m\wt{w}=\wt{v_j}$ (if there is no edge outside $T$ in $Y$, then there is no $\f_m$). 
\end{enumerate}
For $1\leq i\leq n$, let $G_i$ be the $G$-stabilizer of the vertex $\wt{v_i}$. Then,
\[
G=\left\langle \bigcup_{i=1}^n(G_i\cup\left\{\f_j^i:1\leq j\leq k_i\right\})\bigcup \left\{\f_m:1\leq m\leq l\right\} \right\rangle
\]
In particular, if $G_i$'s are finitely generated, then $G$ is also finitely generated.
\end{theorem}

\begin{proof}
Let $S:=\bigcup_{i=1}^n(G_i\cup\{\f_j^i:1\leq j\leq k_i\})\bigcup \{\f_m:1\leq m\leq l\}$ and let $H$ be the subgroup of $G$ generated by $S$. Clearly, $H\subseteq G$. It remains to show that $G\subseteq H$. Let $\f\in G$. Fix the vertex $\wt{v}_1\in \wt{T}$ (we can choose any other vertex also). If $\f\wt{v}_1=\wt{v}_1$ then $\f\in G_1\subset H$, and hence we are done. Suppose $\f\wt{v}_1\neq \wt{v}_1$. Then, $\f\wt{v}_1$ is a vertex of $X$ which is not in $\wt{T}$. Let $\alpha$ be an edge path in $X$ joining $\wt{v}_1$ and $\f\wt{v}_1$. We show that each vertex of $\alpha$ which is not in $\wt{T}$ is of the form $\h\wt{v_i}$ for some $\h\in H$ and $1\leq i\leq n$. This will, in particular, prove that $\f\in H$, and hence we are done. We consider the following cases:

{\bf Case 1.} Suppose $\alpha$ is an edge, and let $\alpha=\wt{\xi}$. Then, there is a loop at the vertex $v_1$ in $Y$, say $\xi_j^1$. Let $\wt{\xi}_j^1$ be the lift of $\xi_j^1$ as in (i). If $\wt{\xi}_j^1=\wt{\xi}$, then $\f v_1=\f_j^1v_1$. This implies that $\f\in H$ and we are done. If not, there exists $\f'\in G$ such that $\f'\wt{\xi}_j^1=\wt{\xi}$. There are two possibilities. Suppose $\f'\wt{v_1}=\wt{v_1}$ and $\f'\f_j^1\wt{v_1}=\f\wt{v_1}$, then $\f'\in G_1\subset H$. As $\f^{-1}\f'\f_j^1\in G_1$ and $\f_j^1,\f'\in H$, we have $\f\in H$. A similar computation shows that if $\f'\wt{v_1}=\f\wt{v_1}$ and $\f'\f_j^1\wt{v_1}=\wt{v_1}$, then $\f\in H$. Hence, in this case, $\f\in H$.

{\bf Case 2.} Suppose $\alpha$ is not an edge. Let $\wt{\xi}$ be any edge of $\alpha$ joining two vertices, say, $\wt{v}$ and $\wt{w}$. There are two subcases to be considered.

{\bf Subcase 2(a).}  Suppose $\wt{v}=\wt{v}_i\in \wt{T}$ for some $i$ and $\wt{w}\notin \wt{T}$. Then, there exists $1\leq j\leq n$ such that $\wt{w}=\f'\wt{v_j}$ for some $\f'\in G$. If $j=i$, then as in Case 1, $\f'\in H$. If $i\neq j$, then there is an edge in $Y$ joining $v_i$ and $v_j$, say $\xi_m$. There are two possibilities. First we  assume that $\xi_m\in T$. Let $\wt{\xi}_m'$ be the lift of $\xi_m$ in $\wt{T}$ joining the vertices $\wt{v_i}$ and $\wt{v_j}$. Then, there exists $\f''\in G$ such that $\f''\wt{\xi}_m'=\wt{\xi}$. This implies that $\f''\in G_i$, which in turn implies that $\f'\in H$.

Now, suppose that $\xi_m\notin T$. By our assumption (ii), let $\wt{\xi}_m$ be a lift of $\xi_m$ such that exactly one vertex of $\wt{\xi}_m$ is either $\wt{v_i}$ or $\wt{v_j}$. Suppose that vertex is $\wt{v_i}$. If $\wt{\xi}_m=\wt{\xi}$, by our hypothesis, there exists $\f_m\in S$ such that $\f_m\f'\wt{v}_j=\wt{v}_j$. This implies that $\f_m\f'\in G_j\subset H$. Since $\f_m\in S\subset H$, $\f'\in H$. Suppose $\wt{\xi}_m\neq\wt{\xi}$ and the other vertex of $\wt{\xi}_m$ is $\f'''\wt{\xi_j}$ for some $\f'''\in G$. By our assumption (2), there exists $\f_m'\in S$ such that $\f_m'\f'''\wt{v_j}=\wt{v_j}$. This implies that $\f'''\in H$. Finally, there exists $\f_1\in G$ such that $\f_1\wt{\xi}_m=\wt{\xi}$. Then, $\f_1\wt{v_i}=\wt{v_i}$ and $\f_1\f'''\wt{v_j}=\f'\wt{v_j}$. This, in turn, implies that $\f'\in H$. A similar computation can be done when the vertex of $\wt{\xi}_m$ in $\wt{T}$ is $\wt{v_j}$.

{\bf Subcase 2(b).} Before proceeding further, we observe the following. Suppose $\wt{\xi}_1$ and $\wt{\xi}_2$ are two distinct edges of $\alpha$ and $\wt{\xi}_1\ast\wt{\xi}_2$ is the concatenation of these two edges such that one vertex of $\wt{\xi}_1$ is in $\wt{T}$ and both the vertices of $\wt{\xi}_2$ are not in $\wt{T}$. Let $\wt{v}_1'$ and $\wt{v}_2'$ be the concatenating vertex and the other vertex of $\wt{\xi}_2$, respectively. By Subcase 2(a), there exists $\h\in H$ and $1\leq i\leq n$ such that $\wt{v}_1'=\h\wt{v}_i$. Suppose $\wt{v}_2=\f'\wt{v}_j'$ for some $\f'\in G$ and $1\leq j\leq n$. Then, $\h^{-1}\wt{\xi}_2$ is the edge joining the vertices $\wt{v}_i$ and $\h^{-1}\f'\wt{v}_j$. Then, again by doing the same computation as in Subcase 2(a), $\h^{-1}\f'\in H$ and hence $\f'\in H$. Thus, by induction, we can show that if $\beta$ is a subpath in $\alpha$ such that only one vertex of $\beta$ is in $\wt{T}$, then all other vertices of $\beta$ are of the form $\h\wt{v}_i$, where $\h\in H$ and $1\leq i\leq n$. 

Now, suppose both the vertices of $\wt{\xi}$ are not in $\wt{T}$. Let, for $1\leq j\leq n$, $\wt{w}=\f'\wt{v_j}$ for some $\f'\in G$. Then, there exists a vertex $\wt{v}_p\in \wt{T}\cap \alpha$ and a subpath $\gamma$ of $\alpha$ joining the vertices $\wt{v}_p$ and $\wt{v}$ such that all the vertices of $\gamma$ other than $\wt{v}_p$ are not in $\wt{T}$. Thus, from the above discussion, $\wt{v}=\h'\wt{v}_i$ for some $\h'\in H$ and $1\leq i\leq n$. Then, $\h'^{-1}\wt{\xi}$ is an edge joining $\wt{v}_i$ and $\h'^{-1}\f'\wt{v}_j$. Thus, by Subcase 2(a), we can see that $\f'\in H$, and hence we are done.
\end{proof}

\subsection{The symplectic criteria for abelian covers}
\label{subsec:symp_criteria}
The action of $\map(S_h)$ on $H_1(S_h;\Z)$ induces a representation $\Psi:\map(S_h)\to \symp(2h,\Z)$ known as the \textit{symplectic representation} of $\map(S_h)$. The subgroup $\ker \Psi$, denoted by $\I(S_h)$, is known as the \textit{Torelli group} of $S_h$. For an integer $k\geq 2$, the kernel of the composition, denoted by $\map(S_h)[k]$, $\Psi_k:\map(S_h)\to \symp(2h,\Z)\to \symp(2h,\Z_k)$ is known as the \textit{level $k$ congruence subgroup} of $\map(S_h)$. Let $A$ be a finite abelian group acting freely on a closed surface $S_g$ with orbit space $\orb_A\approx S_h$, where $h=1+(g-1)/|A|\geq 2$. The cover $p:S_g \to S_h$ induces a short-exact sequence~\cite[Proposition 1.40]{hatcher_book}
\[
1\longrightarrow \pi_1(S_g)\xlongrightarrow{p_{\ast}} \pi_1(S_h)\xlongrightarrow{\eta} A \longrightarrow 1,
\]
where the map $\eta$ is known as a surface kernel map. 
For $i=1,2$ and $\eta_i:\pi_1(S_h)\to A$, we say that $\eta_1$ is equivalent to $\eta_2$ if there are automorphisms $\delta\in \aut(\pi_1(S_h))$ and $\omega\in \aut(A)$ such that $\eta_2=\omega \circ\eta_1\circ\delta$. The equivalence classes of surface kernel maps correspond to the isomorphism classes of finite-sheeted abelian regular covers (or conjugacy classes of free finite abelian group actions)~\cite[Proposition 2.2]{broughton91}.

Since $p$ is an abelian cover, $\eta$ factor through the \textit{induced map} $\eta':H_1(S_h;\Z)\to A$. Therefore, we also have the following exact sequence
\[
H_1(S_g;\Z)\xlongrightarrow{p_{\#}} H_1(S_h;\Z)\xlongrightarrow{\eta'} A \longrightarrow 0.
\]
For $F\in \map(S_h)$, let $F_{\#}$ be the induced automorphism of $H_1(S_h;\Z)$. We identify $F_{\#}$ with the symplectic matrix $\Psi(F)$ and $H_1(S_h;\Z)$ with $\Z^{2h}=\langle e_1,e_2,\dots,e_{2h} \rangle$, where $e_i$ is vector with $1$ at the $i^{th}$ place and $0$'s elsewhere. We know that $F\in \lmap_p(S_h)$ if and only if $F_{\ast}(p_{\ast}(\pi_1(S_g)))=p_{\ast}(\pi_1(S_g))$~\cite[Proposition 1.33]{hatcher_book}. Since $p_*(\pi_1(S_h))=\ker \eta$, $F\in \lmap_p(S_h)$ if and only if $F_*(\ker \eta)=\ker \eta$. Since $p$ is an abelian cover, we have $F\in \lmap_p(S_h)$ if and only if $F_{\#}(\ker \eta')=\ker \eta'$. For $F\in \lmap_p(S_h)$, we provide necessary and sufficient conditions on the entries of the matrix $\Psi(F)$ for elementary abelian groups with invariant factors of prime order. These conditions will be called the \textit{symplectic criteria}.

For a prime number $k$, we fix an elementary abelian $k$-group $\Z_k^r$ of rank $r\leq 2h$. Furthermore, we fix the following standard presentations $\pi_1(S_h)=\langle \alpha_1,\beta_1,\dots,\alpha_h,\beta_h\mid \textstyle\prod_{i=1}^h[\alpha_i,\beta_1]=1\rangle$ and $\Z_k^r=\langle c_1,c_2,\dots,c_r\mid c_i^k=[c_i,c_j]=1~\forall~i,j\rangle$, where $[x,y]$ denotes the commutator $xyx^{-1}y^{-1}$ of $x$ and $y$. The following result due to Broughton~\cite[Theorem 3.1]{broughton07} describes all isomorphism classes of $\Z_k^r$-covers $S_g\to S_h$.

\begin{theorem}
\label{thm:conj_elem_covers}
There exists an integer $r/2\leq K\leq \min(h,r)$ such that any surface kernel map $\pi_1(S_h)\to \Z_k^r$ is $\aut(\pi_1(S_h))\times \aut(\Z_k^r)$-equivalent to one of the following:
\[
\eta_K:=
\begin{cases}
\eta_K(\alpha_i)=c_i & i\leq K\\
\eta_K(\alpha_i)=0 & i>K\\
\eta_K(\beta_i)=c_{i+K} & i\leq r-K\\
\eta_K(\beta_i)=0 & i>r-K.
\end{cases}
\]
\end{theorem}

We derive the symplectic criteria for the map described in Theorem~\ref{thm:conj_elem_covers} with corresponding cover $p_K$. We have the following result.

\begin{theorem}
\label{thm:sym_criteria}
For a prime number $k$, let $p_K:S_g \to S_h$ be a regular $\Z_k^r$-cover. For $F\in \map(S_h)$ and $\Psi:\map(S_h)\to \symp(2h;\Z)$, let $\Psi(F)=(a_{ij})_{2h\times 2h}$. Then $F \in \lmap_{p_K}(S_h)$ if and only 
for $j\in \{2K+1,2K+3,\dots,2h-1\}\cup \{2(r-K+1),2(r-K+2),\dots,2K,\dots,2h\}$, we have $a_{ij}\equiv 0\pmod k$ for every $i\in \{1,2,\dots,2(r-K)+1,2(r-K)+3,\dots,2K-1\}$. Furthermore, $\map(S_h)[k]\subset\lmap_p(S_h)$.
\end{theorem}

\begin{proof}
Let $H_1(S_h;\Z)=\langle a_1,b_1,\dots,a_h,b_h\rangle$. For the map $\eta_K:\pi_1(S_h)\to \Z_k^r$ described in Theorem~\ref{thm:conj_elem_covers}, the induced map $\eta_K':H_1(S_h;\Z)\to \Z_k^r$ is given as
\[
\eta_K':=
\begin{cases}
\eta_K'(a_i)=c_i & i\leq K\\
\eta_K'(a_i)=0 & i>K\\
\eta_K'(b_i)=c_{i+K} & i\leq r-K\\
\eta_K'(b_i)=0 & i>r-K.
\end{cases}
\]
It follows that
\[
\ker \eta_K'=\langle ka_1,kb_1,\dots, ka_{r-K},kb_{r-K},ka_{r-K+1}, b_{r-K+1},\dots, ka_K,b_K,a_{K+1},b_{K+1},\dots,a_h,b_h\rangle.
\]
We have $F\in \lmap_{p_K}(S_h)$ if and only if $F_{\#}(\ker \eta_K')=\ker \eta_K'$. Since $H_1(S_h;\Z)=\langle e_1,e_2,\dots,e_{2h} \rangle$, without loss of generality, we assume that $a_{i'}=e_{2i'-1}$ and $b_{i'}=e_{2i'}$ for $1\leq i'\leq h$. For $j\in \{2K+1,2K+3,\dots,2h-1\}$, we have $\Psi(F)e_j=\textstyle\sum_{i=1}^{2h}a_{ij}e_i\in \ker \eta_K'$ if and only if $a_{ij}\equiv 0 \pmod k$ for all $i\in \{1,2,\dots,2(r-K)+1,2(r-K)+3,\dots,2K-1\}$. Similarly, for $j\in \{2(r-K+1),2(r-K+2),\dots,2K,\dots,2h\}$, we have $\Psi(F)e_j=\textstyle\sum_{i=1}^{2h}a_{ij}e_i\in \ker \eta_K'$ if and only if $a_{ij}\equiv 0 \pmod k$ for all $i\in \{1,2,\dots,2(r-K)+1,2(r-K)+3,\dots,2K-1\}$. Now, it follows that $\map(S_h)[k]\subset\lmap_p(S_h)$.
\end{proof}

\begin{rem}
\label{rem:sym_criteria}
Let $A=\oplus_{i=1}^{r}\Z_{n_i}$ be an abelian group, where $n_i\mid n_{i+1}$. Given a finite-sheeted regular abelian cover $p:S_g\to S_h=S_g/A$, for $F\in \lmap_p(S_h)$, one can derive the necessary and sufficient conditions on the entries of matrix $\Psi(F)$ from $F_{\#}(\ker \eta')=\ker \eta'$, where $\eta:\pi_1(S_h)\to A$ is a surface kernel map corresponding to the cover $p$ and $\eta':H_1(S_h;\Z)\to A$ is the induced map. Since elements of $\map(S_h)[n_r]$ acts trivially on $H_1(S_h;\Z_{n_r})$, we have $\map(S_h)[n_r]\subset \lmap_p(S_h)$.
\end{rem}

\subsection{Reducing the action of $\lmap_p(S_h)$ on the curve graph to a finite graph}
\label{subsec:action_infinite_finite}
The nonseparating curve graph $\N(S_h)$ is defined as follows: the vertices of $\N(S_h)$ are nonseparating essential simple closed curves on $S_h$. Two vertices $c_1$ and $c_2$ are connected by an edge in $\N(S_h)$ if and only if $i(c_1,c_2)=1$, where $i(c_1,c_2)$ is the minimal number of intersections between any representatives of isotopy classes $c_1$ and $c_2$ known as their \textit{geometric intersection number}. It is known that $\N(S_h)$ is connected~\cite[Lemma 2]{lickorish64} (see~\cite[Theorem 4.4]{primer} and~\cite[Lemma A.2]{putman07} also). Let $V(\N(S_h))$ and $E(\N(S_h))$ be the set of vertices and the set of edges of $\N(S_h)$, respectively. By classification of surfaces, $\map(S_h)$ acts transitively on $V(\N(S_h))$ and $E(\N(S_h))$. Let $p$ be an abelian cover of $S_h$. Since $[\map(S_h):\lmap_p(S_h)]<\infty$, the number of orbits of $\lmap_p(S_h)$-action on $V(\N(S_h))$ and $E(\N(S_h))$ is bounded above by $[\map(S_h):\lmap_p(S_h)]$.

Let $k$ be the smallest positive integer such that $\map(S_h)[k]\subset \lmap_p(S_h)$. For $c\in V(\N(S_h))$, let $[c]$ and $[c]_k$ be the homology classes of $c$ in $H_1(S_h;\Z)$ and $H_1(S_h;\Z_k)$, respectively. Consider the equivalence relation on $V(\N(S_h))$ defined as $c_1$ and $c_2$ are equivalent if $[c_1]_k=[c_2]_k$, where $c_1,c_2\in V(\N(S_h))$. The equivalence relation induces a quotient map $\N(S_h)\to \N_k(S_h)$. Therefore, the quotient graph $\N_k(S_h)$ is connected. We will identify the vertices of $\N_k(S_h)$ with primitive elements of $\Z_k^{2h}$. We observe that the group $\Psi_k(\lmap_p(S_h))$ acts naturally on the graph $\N_k(S_h)$. The orbit graphs $\N(S_h)/\lmap_p(S_h)$ and $\N_k(S_h)/\Psi_k(\lmap_p(S_h))$ will be denoted by $\overline{\N(S_h)}$ and $\overline{\N_k(S_h)}$, respectively.

\begin{theorem}
\label{prop:curve_homology_cor}
Let $c,d\in V(\N(S_h))$ such that $[c]_k=[d]_k$. Then there exists a $F \in \map(S_h)[k]$ such that $F(c)=d$. Furthermore, let $\{c_1,c_2\},\{d_1, d_2\}\in E(\N(S_h))$ such that $[c_1]_k=[d_1]_k$ and $[c_2]_k=[d_2]_k$. Then there exists $F'\in \map(S_h)[k]$ such that $F'$ maps the edge $\{c_1,c_2\}$ to the edge $\{d_1, d_2\}$. Moreover, $\overline{\N(S_h)}$ and $\overline{\N_k(S_h)}$ are isomorphic.
\end{theorem}

We provide the proof of Theorem~\ref{prop:curve_homology_cor} through a series of lemmas only in the case of $h=2$ for brevity. Let $\Psi_k:\map(S_2)\to \mathrm{Sp}(4,\Z_k)$ be the $\Z_k$-homology representation of $\map(S_2)$. The symplectic form $\hat{i}:H_1(S_2;\Z)\times H_1(S_2;\Z)\to \Z$ descends to a symplectic form $\hat{i}_k:H_1(S_2;\Z_k)\times H_1(S_2;\Z_k)\to \Z_k$ defined as $\hat{i}_k(x,y):=\hat{i}(u,v)\pmod k$, where $u\equiv x \pmod k$ and $v\equiv y \pmod k$.

\begin{lem}
\label{lem:modk_homo_to_zhomo}
Let $u\in H_1(S_2;\Z)$ be a primitive vector and let $x\equiv u \pmod k$. Let $y\in~ H_1(S_2;\Z_k)$ be a primitive vector such that $\hat{i}_k(x,y)=1$. Then there exists a primitive vector $v\in H_1(S_2;\Z)$ such that $y\equiv v\pmod k$ and $\hat{i}(u,v)=1$. 
\end{lem}

\begin{proof}
There exists a primitive vector $v'\in H_1(S_2;\Z)$ such that $\hat{i}(u,v')=1$. Let $y'\equiv v'\pmod k$. We write $H_1(S_2;\Z_k)=\langle x,y'\rangle \oplus \langle x,y'\rangle^{\perp}$. Since $y\in H_1(S_2;Z_k)$, we have $y=z+z^{\perp}$, where $z\in \langle x,y'\rangle$ and $z^{\perp}\in \langle x,y'\rangle^{\perp}$. Since $1=\hat{i}_k(x,y)=\hat{i}_k(x,z)+\hat{i}_k(x,z^{\perp})$ and $\hat{i}_k(x,z^{\perp})=0$, we have $\hat{i}_k(x,z)=1$. Since $z\in \langle x,y'\rangle$ and $\hat{i}_k(x,y')=1$, we must have $z=mx+y'$ for some $m\in \Z_k$. For $w:=m'u+v'$ and $m'\equiv m \pmod k$, we have $w\equiv mx+y'=z \pmod k$ and $\hat{i}(u,w)=\hat{i}(u,v')=1$. Since $\langle x,z\rangle^{\perp}=\langle x,y'\rangle^{\perp}$ and $\langle u,w\rangle^{\perp}$ maps to $\langle x,z\rangle^{\perp} \pmod k$, we can choose $w'\in \langle u,w\rangle^{\perp}$ such that $w'\equiv z^{\perp}\pmod k$. Then for $v:=w+w'$, we have $v\equiv y \pmod k$ and $\hat{i}(u,v)=1$. If $v=nv''$ for some $n\in \mathbb{N}$ and $v''\in H_1(S_2;\Z)$, then we have $1=\hat{i}(u,v)=n\hat{i}(u,v'')$. This implies that $n=1$ and $v$ is primitive.  
\end{proof}

\begin{lem}
\label{lem:basis_modk_to_z}
Let $\{x_1,y_1,x_2,y_2\}$ be a symplectic basis for $H_1(S_2;\Z_k)$. Then there exists a symplectic basis $\{u_1,v_1,u_2,v_2\}$ of $H_1(S_2;\Z)$ such that, for $i=1,2$, $u_i\equiv x_i\pmod k$ and $v_i\equiv y_i \pmod k$.
\end{lem}

\begin{proof}
We have $H_1(S_2;\Z_k)=\langle x_1,y_1\rangle\oplus \langle x_2,y_2\rangle$ such that $\langle x_2,y_2\rangle=\langle x_1,y_1\rangle^{\perp}$. For $i=1,2$, by Lemma~\ref{lem:modk_homo_to_zhomo}, there exist $u_i,v_i\in H_1(S_2,\Z)$ such that $\hat{i}(u_i,v_i)=1$, $u_i\equiv x_i \pmod k$, and $v_i\equiv y_i\pmod k$. Since $\langle u_1,v_1\rangle$ and $\langle u_2,v_2\rangle$ maps to $\langle x_1,y_1\rangle$ and $\langle x_2,y_2\rangle$ modulo $k$, we have $\{ u_1,v_1,u_2,v_2\}$ is a symplectic basis for $H_1(S_2;\Z)$.
\end{proof}

The following lemma is well-known in the theory of curves on surfaces, which says that a symplectic basis of $H_1(S_2;\Z)$ can be realized by simple closed curves on $S_2$ known as a \textit{geometric symplectcic basis} (follows from the work of Meeks-Patrusky~\cite{meek78}) (see~\cite[Lemma A.3]{putman07} also).

\begin{lem}
\label{lem:basis_z_to_curves}
Let $\{u_1,v_1,u_2,v_2\}$ be a symplectic basis for $H_1(S_2;\Z)$. Then, For $i=1,2$, there exist simple closed curves $a_i$ and $b_i$ on $S_2$ such that $[a_i]=u_i$, $[b_i]=v_i$, $i(a_i,b_i)=1$ and $i(a_1,b_2)=i(a_1,a_2)=i(b_1,b_2)=i(a_2,b_1)=0$.
\end{lem}

\begin{cor}
\label{cor:equal_modk_vertex}
Let $c,d\in V(\N(S_2))$ such that $[c]_k=[d]_k$. Then there exists $F \in \map(S_2)[k]$ such that $F(c)=d$.
\end{cor}

\begin{proof}
For $x_1=[c]_k$, let $B=\{x_1,y_1,x_2,y_2\}$ be a symplectic basis for $H_1(S_2;\Z_k)$. By Lemmas~\ref{lem:modk_homo_to_zhomo}-\ref{lem:basis_modk_to_z}, there exist symplectic bases $ZB_1=\{u_1=[c],v_1,u_2,v_2\}$ and $ZB_2=\{u_1'=[d],v_1',u_2',v_2'\}$ realizing $B$. By Lemma~\ref{lem:basis_z_to_curves}, there exist geometric symplectic basis $C=\{c=a_1,b_1,a_2,b_2\}$ and $D=\{d=a_1',b_1',a_2',b_2'\}$ of $H_1(S_2;\Z)$ realizing $ZB_1$ and $ZB_2$. By classification of surfaces, there exists $F\in \map(S_2)$ such that $F(C)=D$ and $F(c)=d$, where $C$ and $D$ are considered ordered. Since $\Psi_k(F)$ fixes $B$, we have $F\in \map(S_2)[k]$.      
\end{proof}

\begin{cor}
\label{cor:equal_modk_edges}
Let $\{c_1,c_2\},\{d_1, d_2\}\in E(\N(S_2))$ such that $[c_i]_k=[d_i]_k$, where $i=1,2$. Then there exists $F\in \map(S_2)[k]$ such that $F$ maps $\{c_1,c_2\}$ to $\{d_1, d_2\}$. 
\end{cor}

\begin{proof}
For $x_1=[c_1]_k$ and $y_1=[c_2]_k$, let $B=\{x_1,y_1,x_2,y_2\}$ be a symplectic basis of $H_1(S_2,\Z_k)$. By Lemmas~\ref{lem:modk_homo_to_zhomo}-\ref{lem:basis_modk_to_z}, there exist symplectic bases $ZB_1=\{[c_1],[c_2],u_2,v_2\}$ and \linebreak $ZB_2=\{[d_1],[d_2],u_2',v_2'\}$ of $H_1(S_2;\Z)$ realizing $B$. By Lemma~\ref{lem:basis_z_to_curves}, there exist geometric symplectic bases $C=\{c_1,c_2,a_2,b_2\}$ and $D=\{d_1,d_2,a_2',b_2'\}$ realizing $ZB_1$ and $ZB_2$, respectively. By classification of surfaces, there exists $F \in \map(S_2)$ such that $F(C)=D$, where $C$ and $D$ are considered ordered. Since $\Psi_k(F)$ fixes $B$, we have $F \in \map(S_2)[k]$ and maps $\{c_1,c_2\}$ to $\{d_1,d_2\}$.
\end{proof}

\begin{lem}
\label{lem:graph_corres}
The graphs $\overline{\N(S_2)}$ and $\overline{\N_k(S_2)}$ are isomorphic.
\end{lem}

\begin{proof}
Define the map $\psi:V(\overline{\N(S_2)})\to V(\overline{\N_k(S_2)})$ as $\psi(O_c)=O_{[c]_k}$ for $c\in V(\N(S_2))$, where $O_c$ denote the orbit of the vertex $c\in V(\N(S_2))$ under the action $\lmap_p(S_2)$ on $\N(S_2)$. For $O_c=O_d$, there exists $F\in \lmap_p(S_2)$ such that $F(c)=d$. Then $\Psi_k(F)[c]_k=[d]_k$, that is, $O_{[c]_k}=O_{[d]_k}$. This shows that the map $\psi$ is well-defined.

First, we show that $\psi$ is bijective. Since the vertices of $\N_k(S_2)$ are identified with primitive elements of $\Z_k^{2h}$, $\psi$ is surjective. For $O_c,O_d\in V(\overline{\N(S_2)})$, let $\psi(O_c)=\psi(O_d)$, that is, $O_{[c]_k}=O_{[d]_k}$. Then there exists a $F\in \lmap_p(S_2)$ such that $\Psi_k(F)[c]_k=[d]_k$. Since $[F(c)]_k=[d]_k$, by Corollary~\ref{cor:equal_modk_vertex}, there exists $F'\in \map(S_2)[k]$ such that $F'F(c)=d$. Since $F'F\in \lmap_p(S_2)$, we have $O_c=O_d$, and consequently, $\psi$ is injective.

Now we show that $\psi$ induces a bijection $\theta:E(\overline{\N(S_2)})\to E(\overline{\N_k(S_2)})$ given by $\theta(O_{\{c,d\}})=O_{\{[c]_k,[d]_k\}}$, where $O_{\{c,d\}}$ denote the orbit of the edge $\{c,d\}\in E(\N(S_2))$ under the action of $\lmap_p(S_2)$ on $\N(S_2)$. As before, it is clear that $\theta$ is a well-defined surjective map. We show that $\theta$ is injective. For $\{c_1,c_2\},\{d_1,d_2\}\in E(\N(S_2))$, let $O_{\{[c_1]_k,[c_2]_k\}}=O_{\{[d_1]_k,[d_2]_k\}}$. There exists $F\in \lmap_p(S_2)$ such that $\Psi_k(F)\{[c_1]_k,[c_2]_k\}=\{[d_1]_k,[d_2]_k\}$, that is, $\{[F(c_1)]_k,[F(c_2)]_k\}=\{[d_1]_k,[d_2]_k\}$. Since $i(c_1,c_2)=i(d_1,d_2)=1$, by Corollary~\ref{cor:equal_modk_edges}, there exists $F'\in \map(S_2)[k]$ such that $F'$ maps $\{F(c_1),F(c_2)\}$ to $\{d_1,d_2\}$. Since $F'F\in \lmap_p(S_2)$ and $F'F$ maps $\{c_1,c_2\}$ to $\{d_1,d_2\}$, we have $O_{\{c_1,c_2\}}=O_{\{d_1,d_2\}}$, that is, $\theta$ is injective. Hence $\overline{\N(S_2)}$ and $\overline{\N_k(S_2)}$ are isomorphic.
\end{proof}

\subsection{Stabilizer of a nonseparating simple closed curve in $\lmap_p(S_h)$}
For a group $G$ acting on a set $X$, the stabilizer of $x\in X$ in $G$ will be denoted by $\stab_G(x)$. In this subsection, for an abelian cover $p$ of $S_h$, we describe the stabilizer $\stab_{\lmap_p(S_h)}(\gamma)$ of $\gamma\in V(\N(S_h))$ in $\lmap_p(S_h)$. Since $[\stab_{\map(S_h)}(\gamma):\stab_{\lmap_p(S_h)}(\gamma)]\leq [\map(S_h):\lmap_p(S_h)]<\infty$, $\stab_{\lmap_p(S_h)}(\gamma)$ is finitely generated. We have the following short-exact sequence:
\[
1 \to \stab_{\I(S_h)}(\gamma)\to \stab_{\map(S_h)}(\gamma)\to \stab_{\symp(2h,\Z)}(\pm[\gamma])\to 1.
\]  
Since $\I(S_h)\subset \lmap_p(S_h)$, we have
\begin{equation}
\label{ses:stab_lmod}
1 \to \stab_{\I(S_h)}(\gamma) \to \stab_{\lmap_p(S_h)}(\gamma) \to \stab_{\Psi(\lmap_p(S_h))}(\pm [\gamma]) \to 1.
\end{equation}
By the symplectic criterion of $\Psi(\lmap_p(S_h))$, one can construct a finite generating set for \linebreak $\stab_{\Psi(\lmap_p(S_h))}(\pm [\gamma])$, and thus, we describe a generating set for the group $\stab_{\I(S_h)}(\gamma)$. Let $S_{h,\gamma}$ be the closure of the surface obtained by cutting $S_h$ along $\gamma$, and for the resultant boundary components $\gamma^1$ and $\gamma^2$ (see Figure~\ref{fig:s2e_embed_s2}), we have the following short-exact sequence~\cite[Theorem 3.18]{primer}:
\[
1\to \langle T_{\gamma^1}T_{\gamma^2}^{-1} \rangle \to \map(S_{h,\gamma})\to \stab_{\map(S_h)}(\gamma)\to 1.
\]
Let $\I(S_{h,\gamma})$ be the kernel of the composition $\map(S_{h,\gamma})\to \stab_{\map(S_h)}(\gamma)\to \symp(2h,\Z)$. Thus, we have the following short-exact sequence:
\[
1\to \langle T_{\gamma^1}T_{\gamma^2}^{-1} \rangle \to \I(S_{h,\gamma})\to \stab_{\I(S_h)}(\gamma)\to 1.
\]
Let $S$ be a closed subsurface contained in $S_{h,\gamma}$ such that $S_{h,\gamma}\setminus \text{Int}(S)$ is a pair of pants (see Figure~\ref{fig:s2e_embed_s2}), where $\text{Int}(S)$ is the interior of $S$. Let $\widehat{S_{h,\gamma}}$ be the surface obtained from $S_{h,\gamma}$ by capping one of the boundary components by a closed disk. Let $\I(S)$ be the Torelli group of surface $S$. The following result due to Putman~\cite[Theorem 4.1]{putman07} describes the group $\I(S_{h,\gamma})$.

\begin{theorem}
\label{thm:stab_genset}
We have $\I(S_{h,\gamma})=[\pi,\pi]\rtimes \I(S)$, where $\pi=\pi_1(\widehat{S_{h,\gamma}})$. Furthermore, the surjection $\I(S_{2,\gamma})\to \stab_{\I(S_2)}(\gamma)$ restricts to an isomorphism $[\pi,\pi]\to \stab_{\I(S_2)}(\gamma)$.
\end{theorem}

\begin{rem}
\label{rem:johnson_tom_result}
We note that to write a generating set for $\I(S_{h,\gamma})$, one needs a generating set for $\I(S)$ and $[\pi,\pi]$ both. Due to Johnson~\cite{johnson83}, a finite generating set for $\I(S)$ is known when the genus of $S$ is at least three. Furthermore, due to a result of Tomaszewski~\cite{tom03}, a generating set for $[\pi,\pi]$ is also known. A geometric proof of Tomaszewski's result is given by Putman~\cite[Theorem A]{putman22}.
\end{rem}

\subsection{The algorithm}
\label{sec:algo}
We now describe our algorithm.
\begin{algo}
\label{algo}
For $g\geq 2$, given an abelian cover $p:S_g \to S_h$, we follow the following steps to derive a generating set for $\lmap_p(S_h)$. Let $k$ be the smallest positive integer such that $\map(S_h)[k]\subset \lmap_p(S_h)$.
\begin{enumerate}[Step 1:]
\item Use the symplectic criterion (see Theorem~\ref{thm:sym_criteria} and Remark~\ref{rem:sym_criteria}) to get the description of symplectic matrices of $\Psi_k(\lmap_p(S_h))$.

\item For $\lmap_p(S_h)$-action on the curve graph $\N(S_h)$, compute the structure of the finite graph $\overline{\N_k(S_h)}$.

\item Choose one representative in $\N(S_h)$ for each vertex and edge orbit of $\Psi_k(\lmap_p(S_h))$-action on $\N_k(S_h)$ (see Theorem \ref{prop:curve_homology_cor}).
\label{algostep:choose_curve}

\item Use Theorem \ref{thm:stab_genset} and Remark~\ref{rem:johnson_tom_result} to construct a generating set of $\stab_{\I(S_h)}(\gamma)$ for each $\gamma$ chosen in Step \ref{algostep:choose_curve}.

\item Use the symplectic criterion to construct a generating set for $\stab_{\Psi(\lmap_p(S_h))}(\pm [\gamma])$ for each $\gamma$ chosen in Step \ref{algostep:choose_curve}.

\item Use the short-exact sequence~\ref{ses:stab_lmod} to write a generating set for $\stab_{\lmap_p(S_h)}(\gamma)$ for each $\gamma$ chosen in Step \ref{algostep:choose_curve}.

\item Use Theorem~\ref{thm:genset_gpaction_graph} to obtain a generating set for $\lmap_p(S_h)$ which is a union of \linebreak $\stab_{\lmap_p(S_h)}(\gamma)$, for each $\gamma$ chosen in Step~\ref{algostep:choose_curve}, and finitely many elements corresponding to edge orbits of $\lmap_p(S_h)$-action on $\N(S_h)$.
\label{algostep:final}

\item From the generating set obtained in Step~\ref{algostep:final}, use relations in $\map(S_h)$ to obtain a finite generating set for $\lmap_p(S_h)$.
\label{algostep_finitegen}

\item By lifting the generators of $\lmap_p(S_h)$ under the map $\varphi:\smap_p(S_g)\to \lmap_p(S_h)$ obtained in Step~\ref{algostep_finitegen}, derive a finite generating set for the normalizer of the Deck group.
\end{enumerate}
\end{algo}

\section{Applications of the algorithm}
\label{sec:app}
\subsection{A finite generating set for $\lmap_k(S_2)$}
\label{subsec:cyclic_liftable}
For $g\geq 2$, consider the cover $p_k:S_{k(g-1)+1}\to S_g$. The liftable mapping class group of cover $p_k$ will be denoted by $\lmap_k(S_g)$. In~\cite[Theorem 3.9]{dhanwani21}, a finite generating set for $\lmap_k(S_g)$ has been obtained for $g\geq 3$. Furthermore, a finite generating set for $\lmap_2(S_2)$ has also been derived~\cite[Corollary 5.5]{dhanwani21}. In this subsection, using Algorithm~\ref{algo}, for a prime number $k$ we derive a finite generating set for $\lmap_k(S_2)$. As a consequence, we will recover the generating set of $\lmap_2(S_2)$ obtained in~\cite[Corollary 5.5]{dhanwani21} and give an explicit finite generating set for $\lmap_3(S_2)$.

For Step 1 of Algorithm~\ref{algo}, we must
derive the symplectic criterion for $\lmap_k(S_2)$. The following characterization of symplectic matrices was obtained in~\cite[Theorem 2.2]{dhanwani22}.

\begin{prop}
\label{prop:sym_mat_s2}
For $F\in \map(S_g)$, the following statements are equivalent.
\begin{enumerate}[(i)]
\item $F\in \lmap_k(S_g)$.
\item $\Psi(F)=(a_{ij})_{2g\times 2g}$, where $k\mid a_{2i}$, for $1\leq i\leq 2g$ and $i\neq 2$, and $\gcd(a_{22},k)=1$.
\item $\Psi_k(F)=(b_{ij})_{2g\times 2g}$, where $b_{2i}=0$, for $1\leq i\leq 2g$ and $i\neq 2$, and $b_{22}\in \Z_k^{\times}$.
\end{enumerate}
\end{prop}

For Step 2 of Algorithm~\ref{algo}, in the view of Lemma~\ref{lem:graph_corres}, we consider the action of $\Psi_k(\lmap_k(S_2))$ on $\N_k(S_2)$ to determine the finite graph $\overline{\N_k(S_2)}$. Let $\{e_1,e_2,e_3,e_4\}$ be the standard symplectic basis of $H_1(S_2;\Z_k)$ identified with $\Z_k^4$, where $e_j$ has $1$ at the $j^{th}$ coordinate and $0$ elsewhere. For $\ell\in \Z_k$, the multiplicative inverse of $\ell$ in $\Z_k$ will be denoted by $\bar{\ell}$.

\begin{lem}
\label{lem:vertices_modk_graph}
For a prime number $k$, the graph $\overline{\mathcal{N}_k(S_2)}$ has three vertices corresponding to orbits $O_{e_1}$, $O_{e_2}$, and $O_{e_3}$.
\end{lem}

\begin{proof}
Let $x=(i_1,i_2,i_3,i_4)\in \Z_k^4$ be a primitive vector. Since $x$ is primitive, not all $i_j$'s are zero. We use the description of symplectic matrices in $\lmap_k(S_2)$ described in Proposition~\ref{prop:sym_mat_s2}.

\textbf{Case 1:} $i_2\neq 0$. For the matrix
\[
A=\begin{pmatrix}
\bar{i_2} & i_1 & -\bar{i_2}i_4 & \bar{i_2}i_3 \\ 
0 & i_2 & 0 & 0 \\ 
0 & i_3 & 1 & 0 \\ 
0 & i_4 & 0 & 1
\end{pmatrix},
\]
we have $A\in \Psi_k(\lmap_k(S_2))$ and $Ae_2=x$.

\textbf{Case 2:} $i_2=0$ and $i_3\neq 0$. In this case, for
\[
A=\begin{pmatrix}
1 & 0 & i_1 & 0 \\ 
0 & 1 & 0 & 0 \\ 
0 & 0 & i_3 & 0 \\ 
0 & -\bar{i_3}i_1 & i_4 & \bar{i_3}
\end{pmatrix} ,
\]
we have $A\in \Psi_k(\lmap_k(S_2))$ and $Ae_3=x$.

\textbf{Case 3:} $i_2=i_3=0$ and $i_4\neq 0$. In this case, for
\[
A=\begin{pmatrix}
1 & 0 & i_1 & 0 \\ 
0 & 1 & 0 & 0 \\ 
0 & \bar{i_4}i_1 & 0 & -\bar{i_4} \\ 
0 & 0 & i_4 & 0
\end{pmatrix},
\]
we have $A\in \Psi_k(\lmap_k(S_2))$ and $Ae_3=x$.

\textbf{Case 4:} $i_2=i_3=i_4=0$ and $i_1\neq 0$. In this case, for
\[
A=\begin{pmatrix}
i_1 & 0 & 0 & 0 \\ 
0 & \bar{i_1} & 0 & 0 \\ 
0 & 0 & 1 & 0 \\ 
0 & 0 & 0 & 1
\end{pmatrix},
\]
we have $A\in \Psi_k(\lmap_k(S_2))$ and $Ae_1=x$.

Now, it follows that $O_{e_1}$, $O_{e_2}$, and $O_{e_3}$ are three distinct orbits of $\Psi_k(\lmap_k(S_2))$-action on $\N_k(S_2)$. 
\end{proof}

The three vertices of the graph $\overline{\N_k(S_2)}$ corresponding to orbits $O_{e_1}$, $O_{e_2}$, and $O_{e_3}$ will be denoted by $\overline{e_1}$, $\overline{e_2}$, and $\overline{e_3}$, respectively. The zero matrix and the identity matrix of  size $n\times n$ will be denoted by $O_{n\times n}$ and $I_{n\times n}$, respectively. 

\begin{lem}
\label{lem:edhes_modk_graph}
For a prime number $k$, the graph $\overline{\N_k(S_2)}$ has an edge between $\overline{e_1}$ and $\overline{e_2}$, an edge between $\overline{e_2}$ and $\overline{e_3}$, a loop at $\overline{e_3}$, and $2(k-1)$ loops at $\overline{e_2}$.
\end{lem}

\begin{proof}
To determine the set of edges of the graph $\overline{\N_k(S_2)}$, we determine distinct orbits of edges of the graph $\N_k(S_2)$ under the action of $\Psi_k(\lmap_k(S_2))$. We first provide computations for loops at vertices of $\overline{\N_k(S_2)}$ and then for edges between distinct vertices of $\overline{\N_k(S_2)}$. Let $v,v'\in \Z_k^4$ be primitive vectors given as $v=(i_1,i_2,i_3,i_4)$ and $v'=(i_1',i_2',i_3',i_4')$.

\textbf{Loops at $\overline{e_1}$:} From the proof of Lemma~\ref{lem:vertices_modk_graph}, it follows that $v\in O_{e_1}$ if and only if $i_2=i_3=i_4=0$ and $i_1\neq 0$. Since $\hat{i}_k(e_1,v)=0$, $\{e_1,v\}$ is not an edge of $\N_k(S_2)$, and therefore, there are no loops at $\overline{e_1}$.

\textbf{Loops at $\overline{e_3}$:} From the proof of Lemma~\ref{lem:vertices_modk_graph}, it follows that $v\in O_{e_3}$ if and only if $i_2=0$ and either $i_3\neq 0$ or $i_4\neq 0$. Since $e_4\in O_{e_3}$ and $\hat{i}_k(e_3,e_4)=1$, there is a loop at $\overline{e_3}$ corresponding to the orbit $O_{{e_3,e_4}}$. Let $v,v'\in O_{e_3}$ such that $\hat{i}_k(v,v')=1$. We claim that there exists $A\in \Psi_k(\lmap_k(S_2))$ such that $A\{e_3,e_4\}=\{v,v'\}$. Since $v'\in O_{e_3}$, we have $i_2'=0$ and either $i_3'\neq 0$ or $i_4'\neq 0$. Since $\hat{i}_k(v,v')=1$, we have $i_3i_4'-i_4i_3'=1$. For the matrix
\[
A=\begin{pmatrix}
1 & 0 & i_1 & i_1' \\ 
0 & 1 & 0 & 0 \\ 
0 & i_1'i_3-i_1i_3' & i_3 & i_3' \\ 
0 & i_1'i_4-i_4'i_1 & i_4 & i_4'
\end{pmatrix},
\]
we have $A\in \Psi_k(\lmap_k(S_2))$ and $A\{e_3,e_4\}=\{v,v'\}$. Hence, there is exactly one loop at $\overline{e_3}$ in $\overline{\N_k(S_2)}$ corresponding to the orbit $O_{\{e_3,e_4\}}$.

\textbf{Loops at $\overline{e_2}$:} From the proof of Lemma~\ref{lem:vertices_modk_graph}, it follows that $v\in O_{e_2}$ if and only if $i_2\neq 0$. For $\{e_2,v\}$ to be an edge of $\N_k(S_2)$, we must have $i_1=1$. Since $\Sl(2,\Z_k)$ acts transitively on the set of primitive vectors, if $(i_3,i_4)\neq (0,0)$, then there exists $A'\in \Sl(2,\Z_k)$ such that $A'(i_3,i_4)=(1,0)$. Then, for the matrix
\[
A=\begin{pmatrix}
I_{2\times 2} & O_{2\times 2} \\ 
O_{2\times 2} & A'
\end{pmatrix},
\] 
we have $A\in \Psi_k(\lmap_k(S_2))$ and $Av=(1,i_2,1,0)$. For $1\leq i_2,i_2'<k$, $v=(1,i_2,0,0)$ and $v'=(1,i_2',1,0)$, it can be seen that $O_{\{e_2,v\}}$ and $O_{\{e_2,v'\}}$ are distinct orbits of edges corresponding to $2(k-1)$ loops at $\overline{e_2}$.

\textbf{Edges between $\overline{e_1}$ and $\overline{e_2}$:} Since $\hat{i}_k(e_1,e_2)=1$, $\{\overline{e_1},\overline{e_2}\}$ is an edge of $\overline{\N_k(S_2)}$. Let $v\in O_{e_1}$ and $v'\in O_{e_2}$ such that $\hat{i}_k(v,v')=1$. Then we have $i_2=i_3=i_4=0$, $i_1\neq 0$, and $i_2'\neq 0$. Since $\hat{i}_k(v,v')=1$, we have $i_1i_2'=1$. Then, for the matrix
\[
A=\begin{pmatrix}
i_1 & i_1' & 0 & 0 \\ 
0 & i_2' & 0 & 0 \\ 
0 & i_3' & 1 & 0 \\ 
0 & i_4' & 0 & 1
\end{pmatrix},
\]
we have $A\{e_1,e_2\}=\{v,v'\}$. Therefore, there is exactly one edge between $\overline{e_1}$ and $\overline{e_2}$ corresponding to the orbit $O_{\{e_1,e_2\}}$.

\textbf{Edges between $\overline{e_1}$ and $\overline{e_3}$:} Let $v\in O_{e_1}$ and $v'\in O_{e_3}$. Then $i_1\neq 0$, $i_2'=i_2=i_3=i_4=0$, and either $i_3'\neq 0$ or $i_4'\neq 0$. Since $\hat{i}_k(v,v')=0$, there are no edges between $\overline{e_1}$ and $\overline{e_3}$.

\textbf{Edges between $\overline{e_2}$ and $\overline{e_3}$:} Since $e_3-e_1\in O_{e_3}$ and $\hat{i}_k(e_2,e_3-e_1)=1$, $\{\overline{e_2},\overline{e_3}\}$ is an edge of $\overline{\N_k(S_2)}$. Let $v\in O_{e_2}$ and $v'\in O_{e_3}$ such that $\hat{i}_k(v,v')=1$. Then we have $-i_2i_1'+i_3i_4'-i_4i_3'=1$, $i_2\neq 0$, $i_2'= 0$, and either $i_3'\neq 0$ or $i_4'\neq 0$. Then there exists $A\in \Psi_k(\lmap_k(S_2))$ such that $A\{e_2,e_3-e_1\}=\{v,v'\}$. For $i_3'\neq 0$, we choose
\[
A=\begin{pmatrix}
\bar{i_2} & i_1 & i_1'+\bar{i_2} & \bar{i_2}i_3\bar{i_3'} \\ 
0 & i_2 & 0 & 0 \\ 
0 & i_3 & i_3' & 0 \\ 
0 & i_4 & i_4' & \bar{i_3'}
\end{pmatrix} 
\]
and for $i_3'=0$, we choose
\[
A=\begin{pmatrix}
\bar{i_2} & i_1 & \bar{i_2}+i_1' & \bar{i_2}i_4\bar{i_4'} \\ 
0 & i_2 & 0 & 0 \\ 
0 & i_3 & 0 & -\bar{i_4'} \\ 
0 & i_4 & i_4' & 0
\end{pmatrix} 
\]
Hence there is exactly one edge between $\overline{e_2}$ and $\overline{e_3}$ corresponding to the orbit $O_{\{e_2,e_3-e_1\}}$.
\end{proof}

For Step 3 of Algorithm~\ref{algo}, we choose curves $a$, $b$, and $c$ in $\N(S_2)$ corresponding to orbits $O_{e_1}$, $O_{e_2}$, and $O_{e_3}$ in $\overline{\N_k(S_2)}$, respectively. For loops in $\overline{\N_k(S_2)}$, we choose edges $\{c,d\}$, $\{b,T_b^{-i_1}(a)\}$, and $\{b,T_b^{-i_2}(c)\}$ in $\N(S_2)$ corresponding to orbits $O_{\{e_3,e_4\}}$, $O_{\{e_2,e_1+i_1e_2\}}$, and $O_{\{e_2,e_1+i_2e_2+e_3\}}$, respectively, where $1\leq i_1,i_2<k$. For edges in $\overline{\N_k(S_2)}$ which are not loops, we choose edges $\{a,b\}$ and $\{b,c\}$ in $\N(S_2)$ corresponding to orbits $O_{\{e_1,e_2\}}$ and $O_{\{e_2,e_3\}}$, respectively. For $F=T_dT_cT_eT_d\in \lmap_k(S_2)$, since $F(e)=c$, we have $\stab_{\lmap_k(S_2)}(c)=F\stab_{\lmap_k(S_2)}(e)F^{-1}$. Therefore, for easier computations, we will derive a generating set for $\stab_{\lmap_k(S_2)}(e)$ instead of $\stab_{\lmap_k(S_2)}(c)$.

In the following lemma, we show that $\stab_{\map(S_2)}(a)\subset \lmap_k(S_2)$. Therefore, it is easier to directly compute a generating set for $\stab_{\lmap_k(S_2)}(a)$. Let $\iota\in \map(S_2)$ be the hyperelliptic involution.

\begin{lem}
\label{lem:stab_a}
We have $\stab_{\lmap_k(S_2)}(a)=\langle T_a,T_c,T_d,T_e,\iota \rangle$.
\end{lem}

\begin{proof}
We recall the short-exact sequence~\cite[Proposition 3.20]{primer} induced by deleting the curve $a$ from $S_2$:
\[
1 \to \langle T_{a}\rangle \to\mathrm{Stab_{\map(S_2)}}(a) \to \map(S_2\setminus a)\cong\map(S_{1,2})\to 1.
\]
For curves shown in Figure \ref{fig:curves_s12}, it is known~\cite[Chapter 4]{primer} that $\map(S_{1,2})=\langle T_{\alpha},T_{\beta},T_{\gamma},\bar{\iota} \rangle$, where $\bar{\iota}\in \map(S_{1,2})$ is the hyperelliptic involution represented by $\pi$-rotation of $S_{1,2}$ as shown in Figure \ref{fig:curves_s12}. 
\begin{figure}[ht]
\centering
\includegraphics[scale=.7]{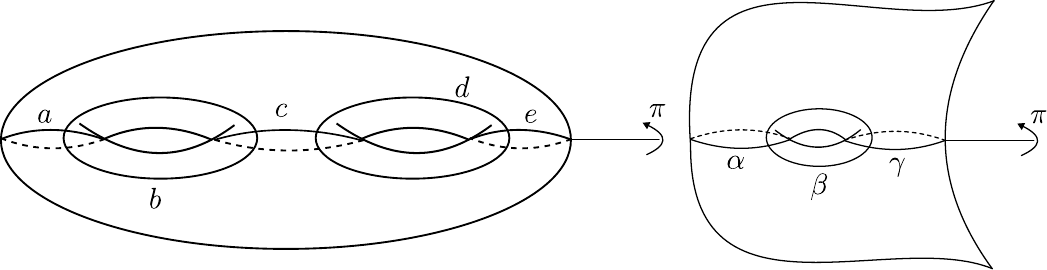}
\caption{The surface $S_{1,2}$ obtained from $S_2$ by deleting the curve $e$.}
\label{fig:curves_s12}
\end{figure}
Since $T_{c} \mapsto T_{\gamma}$, $T_{d}\mapsto T_{\beta}$, $T_{e}\mapsto T_{\alpha}$, and $\iota \mapsto \bar{\iota}$ under the map $\stab_{\map(S_2)}(a)\to \map(S_2\setminus a)$ (see Figure~\ref{fig:curves_s12}), we have $\mathrm{Stab}_{\map(S_2)}(a)=\langle T_{a},T_{c},T_{d},T_{e},\iota\rangle$. From Proposition~\ref{prop:sym_mat_s2}, it follows that $T_{a},T_{c},T_{d},T_{e},\iota\in \lmap_k(S_2)$, we have $\stab_{\map(S_2)}(a)\subset \lmap_k(S_2)$. Hence $\stab_{\lmap_k(S_2)}(a)=\langle T_a,T_c,T_d,T_e,\iota \rangle$.
\end{proof}

In our computations, we will frequently use the following well-known (see~\cite[Section 3.3-3.5]{primer}) relations between Dehn twists.

\begin{lem}
Let $c_1$ and $c_2$ be two simple closed curves on $S$. Then the following holds.
\begin{enumerate}[(i)]
\item If $i(c_1,c_2)=0$, then $T_{c_1}T_{c_2}=T_{c_2}T_{c_1}$.
\item For any $F\in \map(S)$, $FT_{c_1}F^{-1}=T_{F(c_1)}$.
\item If $i(c_1,c_2)=1$, then the following are equivalent.
\begin{enumerate}[(a)]
\item $T_{c_1}T_{c_2}T_{c_1}=T_{c_2}T_{c_1}T_{c_2}$.
\item $T_{c_1}T_{c_2}T_{c_1}^{-1}=T_{c_2}^{-1}T_{c_1}T_{c_2}$.
\item $T_{c_1}^{-1}T_{c_2}T_{c_1}=T_{c_2}T_{c_1}T_{c_2}^{-1}$.
\item $T_{c_1}T_{c_2}(c_1)=c_2$.
\end{enumerate} 
\end{enumerate} 
\end{lem}

For Step 4 of Algorithm~\ref{algo}, we derive a generating set for $\stab_{\I(S_2)}(\gamma)$, where $\gamma\in \{b,c\}$. First, we provide computations for the case $\gamma=e$. We note that $S_{2,e}$ is embedded in $S_2$ as shown in Figure~\ref{fig:s2e_embed_s2}.

\begin{figure}[ht]
\centering
\includegraphics[scale=0.7]{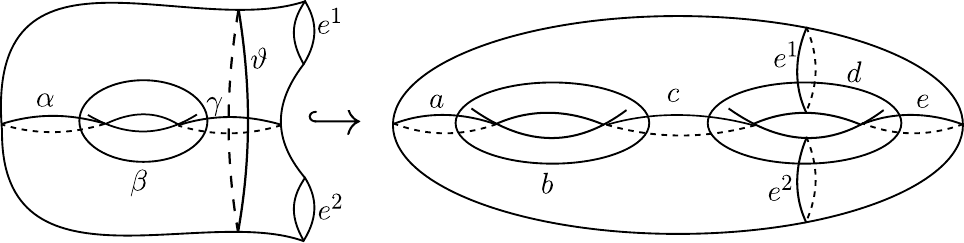}
\caption{An embedding of $S_{2,e}$ in $S_2$.}
\label{fig:s2e_embed_s2}
\end{figure}

\begin{lem}
\label{lem:stab_e_torelli}
We have
\[
\stab_{\I(S_2)}(e)=\langle
T_a^{-m-1}T_c^mT_b^{-1}T_c^n(T_aT_b)^6T_c^{-n}T_bT_c^{-m}T_a^{m+1}\mid m,n\in \Z\rangle.
\]
\end{lem}

\begin{proof}
By Theorem~\ref{thm:stab_genset}, it follows that $\stab_{\I(S_2)}(e)\cong [\pi,\pi]$, where $\pi=\pi_1(\widehat{S_{2,e}})$. Consider the generating set $\langle \tilde{a},\tilde{b}\rangle$ of $\pi$ shown in Figure~\ref{fig:gen_set_pi}.
\begin{figure}[ht]
\centering
\includegraphics[scale=0.7]{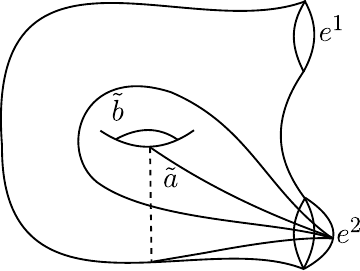}
\caption{Generators of $\pi$.}
\label{fig:gen_set_pi}
\end{figure}
The group $\pi$ is embedded in $\map(S_{2,e})$ via pushing the capped disk of $S_{2,e}$ along loops~\cite[Section 4.2]{primer}. Therefore, the images of $\tilde{a}$ and $\tilde{b}$ in $\map(S_{2,e})$ are $T_{\bar{a}}^{-1}T_{\bar{c}}T_{e^2}^{\ell}$ and $T_{\bar{b}}^{-1}T_{\bar{f}}T_{e^2}^{\ell'}$ for some $\ell,\ell'\in \Z$ (see~\cite[Proposition 2.8]{salter21} also), where curves $\bar{a}$, $\bar{c}$, $\bar{b}$, and $\bar{f}$ are shown in Figure~\ref{fig:nbd_genset_pi}.
\begin{figure}[ht]
\centering
\includegraphics[scale=.7]{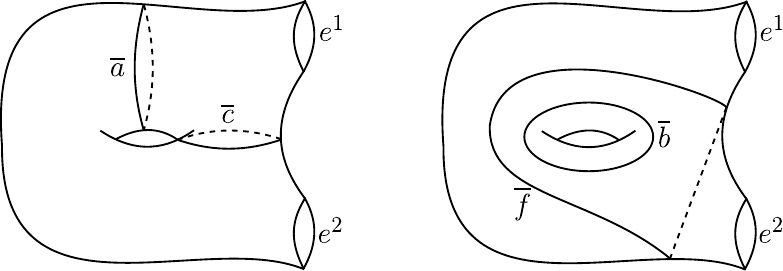}
\caption{The boundary curves of annular neighbouhoods of generators of $\pi$.}
\label{fig:nbd_genset_pi}
\end{figure}
It is known~\cite[Theorem A]{putman22} that $[\pi,\pi]$ is generated by the set $\{\tilde{a}^m\tilde{b}^n[\tilde{a},\tilde{b}]\tilde{b}^{-n}\tilde{a}^{-m}\mid m,n\in \Z\}$. Furthermore, the image of $[\tilde{a},\tilde{b}]\in [\pi,\pi]$ in $S_{2,e}$ is $T_{\vartheta}T_{e^1}^{-1}T_{e^2}$, where $\vartheta$ is the separating curve shown in Figure~\ref{fig:s2e_embed_s2}. Since $T_{e^1}$ and $T_{e^2}$ maps to $T_e$ in $\map(S_2)$ under the map $\map(S_{2,e})\to \stab_{\map(S_2)}(e)$, we have that $[\tilde{a},\tilde{b}]$ maps to $(T_aT_b)^6$ in $\I(S_2)$. Therefore, we have
\[
\stab_{\I(S_2)}(e)=\langle
(T_a^{-1}T_cT_e^{\ell})^m(T_b^{-1}T_fT_e^{\ell'})^n(T_aT_b)^6(T_b^{-1}T_fT_e^{\ell'})^{-n}(T_a^{-1}T_cT_e^{\ell})^{-m}\mid m,n\in \Z\rangle,
\]
where $f$ is shown in Figure~\ref{fig:curves_s2}. Since $T_e$ commutes with $T_a$, $T_b$, $T_c$, $T_f$, and $(T_aT_b)^6$, we have
\[
\stab_{\I(S_2)}(e)=\langle
(T_a^{-1}T_c)^m(T_b^{-1}T_f)^n(T_aT_b)^6(T_b^{-1}T_f)^{-n}(T_a^{-1}T_c)^{-m}\mid m,n\in \Z\rangle.
\]
\begin{figure}[ht]
\centering
\includegraphics[scale=.7]{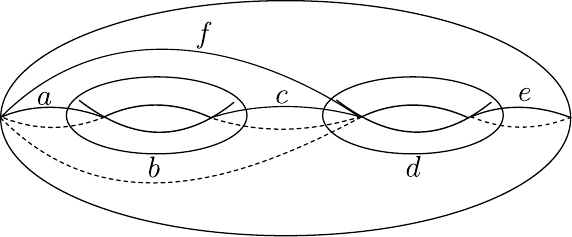}
\caption{The figure shows that $e$ is disjoint from $a,b,c$, and $f$.}
\label{fig:curves_s2}
\end{figure}
Since $f=T_a^{-1}T_b^{-1}(c)$, we have $T_f=T_a^{-1}T_b^{-1}T_cT_bT_a$.

\noindent Since
\begin{align*}
(T_a^{-1}T_c)^m(T_b^{-1}T_f)^n(T_aT_b)^6(T_b^{-1}T_f)^{-n}(T_a^{-1}T_c)^{-m}&=(T_a^{-1}T_c)^mT_f^n(T_aT_b)^6T_f^{-n}(T_a^{-1}T_c)^{-m}\\
&=(T_a^{-1}T_c)^mT_a^{-1}T_b^{-1}T_c^n(T_aT_b)^6T_c^{-n}T_bT_a(T_a^{-1}T_c)^{-m}\\
&=T_a^{-m-1}T_c^mT_b^{-1}T_c^n(T_aT_b)^6T_c^{-n}T_bT_c^{-m}T_a^{m+1},
\end{align*} 
we have
\[
\stab_{\I(S_2)}(e)=\langle
T_a^{-m-1}T_c^mT_b^{-1}T_c^n(T_aT_b)^6T_c^{-n}T_bT_c^{-m}T_a^{m+1}\mid m,n\in \Z\rangle.
\]
\end{proof}

\begin{lem}
\label{lem:stab_b_torelli}
We have
\[
\stab_{\I(S_2)}(b)=\langle T_cT_dT_eT_d^{m+1}T_a^{-1}T_b^{-m-1}T_c^nT_d^{-1}(T_bT_c)^6(T_cT_dT_eT_d^{m+1}T_a^{-1}T_b^{-m-1}T_c^nT_d^{-1})^{-1}
\mid m,n\in \Z\rangle.
\]
\end{lem}

\begin{proof}
For $F=T_cT_bT_dT_cT_eT_d$, since $F(e)=b$, we have $\stab_{\I(S_2)}(b)=F\stab_{\I(S_2)}(e)F^{-1}$. Thus
\[
\stab_{\I(S_2)}(b)=\langle
T_{F(a)}^{-m-1}T_{F(c)}^mT_{F(b)}^{-1}T_{F(c)}^n(T_{F(a)}T_{F(b)})^6T_{F(c)}^{-n}T_{F(b)}T_{F(c)}^{-m}T_{F(a)}^{m+1}\mid m,n\in \Z\rangle.
\]
For $G=T_aT_bT_cT_d$, since $G(a)=b$, $G(b)=c$, and $G(c)=d$, we have
\begin{align*}
&T_{F(a)}^{-m-1}T_{F(c)}^mT_{F(b)}^{-1}T_{F(c)}^n(T_{F(a)}T_{F(b)})^6T_{F(c)}^{-n}T_{F(b)}T_{F(c)}^{-m}T_{F(a)}^{m+1}\\
&=T_cT_bT_dT_cT_eT_dT_a^{-m-1}T_c^mT_b^{-1}T_c^nT_b(T_aT_b)^6(T_cT_bT_dT_cT_eT_dT_a^{-m-1}T_c^mT_b^{-1}T_c^nT_b)^{-1}\\
&=T_cT_dT_eT_a^{-1}(T_aT_bT_cT_d)T_a^{-m-1}T_c^{m+1}T_b^nT_c^{-1}(T_aT_b)^6(T_cT_dT_eT_a^{-1}(T_aT_bT_cT_d)T_a^{-m-1}T_c^{m+1}T_b^nT_c^{-1})^{-1}\\
&=T_cT_dT_eT_d^{m+1}T_a^{-1}T_b^{-m-1}T_c^nT_d^{-1}(T_bT_c)^6(T_cT_dT_eT_d^{m+1}T_a^{-1}T_b^{-m-1}T_c^nT_d^{-1})^{-1}.
\end{align*}
Hence, we have
\[
\stab_{\I(S_2)}(b)=\langle T_cT_dT_eT_d^{m+1}T_a^{-1}T_b^{-m-1}T_c^nT_d^{-1}(T_bT_c)^6(T_cT_dT_eT_d^{m+1}T_a^{-1}T_b^{-m-1}T_c^nT_d^{-1})^{-1}
\mid m,n\in \Z\rangle.
\]
\end{proof}

Let $\Gamma_0(k)$ be the subgroup of $\Sl(2,\Z)$ defined as follows:
\[
\Gamma_0(k)=\left\{
\begin{pmatrix}
a & b \\ 
c & d
\end{pmatrix}\in \Sl(2,\Z)
\Big| c\equiv 0\pmod k 
\right\}.
\]
Let $\phi:\Sl(2,\Z)\to \symp(4,\Z)$ be the injective homomorphism defined as
\[
\phi:
\begin{pmatrix}
a & b \\ 
c & d
\end{pmatrix}
\mapsto
\begin{pmatrix}
a & b & 0 & 0 \\ 
c & d & 0 & 0 \\ 
0 & 0 & 1 & 0 \\ 
0 & 0 & 0 & 1
\end{pmatrix}.
\] 
Let $\s\subset\lmap_k(S_2)$ such that $T_b^k\in \s$ and $\Psi(\s)$ is a finite generating set for $\phi(\Gamma_0(k))$. For Step 5 of Algorithm~\ref{algo}, we compute a generating set for $\stab_{\Psi(\lmap_k(S_2))}(\pm[\gamma])$ for $\gamma=e,b$. For these computations, we define the following matrices:
\begin{equation}
\label{eqn:M_matrices}
M=\begin{pmatrix}
 1 & 0 & 0 & 0 \\ 
 0 & 1 & 0 & 1 \\ 
 -1 & 0 & 1 & 0 \\ 
 0 & 0 & 0 & 1
 \end{pmatrix}, \text{ and }
N=\begin{pmatrix}
1 & 0 & 0 & 0 \\ 
0 & 1 & 1 & 0 \\ 
0 & 0 & 1 & 0 \\ 
1 & 0 & 0 & 1
\end{pmatrix}.
\end{equation}

\begin{lem}
\label{lem:stab_e_matrices}
We have
\[
\stab_{\Psi(\lmap_k(S_2))}(\pm[e])=\langle \phi(\Gamma_0(k)),\Psi(T_a), \Psi(T_c),\Psi(T_e), M^k, -I_{4\times 4} \rangle. 
\]
\end{lem}

\begin{proof}
Since $-I_{4\times 4}([e])=-[e]$ and $-I_{4\times 4}\in \Psi(\lmap_k(S_2))$, we compute $\stab_{\Psi(\lmap_k(S_2))}([e])$. Since $[e]=e_3\in \Z^4$, for any $A\in \stab_{\Psi(\lmap_k(S_2))}(e_3)$, we have $A\in \symp(4,\Z)$ and
\[
A=\begin{pmatrix}
a_{11} & a_{12} & 0 & a_{14} \\ 
a_{21} & a_{22} & 0 & a_{24} \\ 
a_{31} & a_{32} & 1 & a_{34} \\ 
0 & 0 & 0 & 1
\end{pmatrix},
\]
where $a_{21}, a_{31},a_{24}\in k\Z$ and $\gcd(a_{11},k)=\gcd(a_{22},k)=1$. Since
\[
\begin{pmatrix}
a_{11} & a_{12} \\ 
a_{21} & a_{22}
\end{pmatrix}
\in \Gamma_0(k).
\]
there exists $A_1\in \Gamma_0(k)$ such that for $A_2:=\phi(A_1)A$, we have
\[
A_2=\begin{pmatrix}
1 & 0 & 0 & a_{32} \\ 
0 & 1 & 0 & -a_{31} \\ 
a_{31} & a_{32} & 1 & a_{34} \\ 
0 & 0 & 0 & 1
\end{pmatrix}.
\]
Furthermore, we have
\[
A_3:=\Psi(T_e)^{(a_{32}a_{31}-a_{34})}M^{a_{31}}A_2=
\begin{pmatrix}
1 & 0 & 0 & a_{32} \\ 
0 & 1 & 0 & 0 \\ 
0 & a_{32} & 1 & 0 \\ 
0 & 0 & 0 & 1
\end{pmatrix}.
\]
Now, we observe that $A_3=(\Psi(T_e)\Psi(T_a)\Psi(T_c)^{-1})^{-a_{32}}$. Hence
\[
\stab_{\Psi(\lmap_k(S_2))}(\pm[e])=\langle \phi(\Gamma_0(k)),\Psi(T_a), \Psi(T_c),\Psi(T_e), M^k, -I_{4\times 4} \rangle. 
\]
\end{proof}

\begin{lem}
\label{lem:stab_b_matrices}
We have
\[
\stab_{\Psi(\lmap_k(S_2))}(\pm[b])=\langle \Psi(T_b)^k,\Psi(T_d),\Psi(T_e),M^k,N^k,-I_{4\times 4} \rangle. 
\]
\end{lem}

\begin{proof}
Since $-I_{4\times 4}([b])=-[b]$ and $-I_{4\times 4}\in \Psi(\lmap_k(S_2))$, we compute $\stab_{\Psi(\lmap_k(S_2))}([b])$. Since $[b]=e_2\in \Z^4$, for any $A\in \stab_{\Psi(\lmap_k(S_2))}(e_2)$, we have $A\in \symp(4,\Z)$ and
\[
A=\begin{pmatrix}
1 & 0 & 0 & 0 \\ 
a_{21} & 1 & a_{23} & a_{24} \\ 
a_{31} & 0 & a_{33} & a_{34} \\ 
a_{41} & 0 & a_{43} & a_{44}
\end{pmatrix},
\]
where $a_{21}, a_{31}, a_{41}, a_{23} a_{24}\in k\Z$. Since
\[
\begin{pmatrix}
a_{33} & a_{34} \\ 
a_{43} & a_{44}
\end{pmatrix}\in \Sl(2,\Z)
\]
there exists $A_1\in \langle \Psi(T_d),\Psi(T_e)\rangle$ such that
\[
A_1A=
\begin{pmatrix}
1 & 0 & 0 & 0 \\ 
a_{21} & 1 & a_{23} & a_{24} \\ 
-a_{24} & 0 & 1 & 0 \\ 
a_{23} & 0 & 0 & 1
\end{pmatrix}.
\]
By using matrices $\Psi(T_b)$, $M$, and $N$ from Equation~(\ref{eqn:M_matrices}), we have $\Psi(T_b)^{A_2(2,1)}A_2=I_{4\times 4}$, where $A_2=M^{-a_{24}}N^{-a_{23}}A_1A$ (here $A_2(2,1)$ denoted the second entry in first column of $A_2$). Hence, we have
\[
\stab_{\Psi(\lmap_k(S_2))}(\pm[b])=\langle \Psi(T_b)^k,\Psi(T_d),\Psi(T_e),M^k,N^k,-I_{4\times 4} \rangle. 
\]
\end{proof}

It can be seen that $M=\Psi(T_aT_bT_a^{-1}T_cT_b^{-1}T_a^{-1}T_e^{-1})$ and $N=\Psi(T_dT_eT_d)M\Psi(T_dT_eT_d)^{-1}$. For Step 6 of Algorithm~\ref{algo}, we finally write $\stab_{\lmap_k(S_2)}(\gamma)$ for $\gamma=b,e$.

\begin{cor}
\label{cor:stab_e_lmod}
We have
\[
\stab_{\lmap_k(S_2)}(e)=\langle \s\cup \{T_a,T_c,T_e,\iota\}\cup \{T_b^jT_c^{-1}(T_aT_b)^6T_cT_b^{-j}\mid -k< j\leq 0\} \rangle.
\]
\end{cor}

\begin{proof}
From short-exact sequence~(\ref{ses:stab_lmod}) and Lemmas~\ref{lem:stab_e_torelli}-\ref{lem:stab_e_matrices}, it follows that $\stab_{\lmap_k(S_2)}(e)=\langle \s\cup \s'\cup \s''\rangle$, where $$\s'=\{T_a,T_c,T_e,T_aT_bT_a^{-k}T_c^kT_b^{-1}T_a^{-1}T_e^{-k},\iota\}$$ and $$\s''=\{T_a^{-m-1}T_c^mT_b^{-1}T_c^n(T_aT_b)^6T_c^{-n}T_bT_c^{-m}T_a^{m+1}\mid m,n\in \Z\}.$$
We observe that
\begin{align*}
T_aT_bT_a^{-k}T_c^kT_b^{-1}T_a^{-1}T_e^{-k}&= T_a(T_bT_a^{-k}T_b^{-1})(T_bT_c^kT_b^{-1})T_a^{-1}T_e^{-k}\\
&=T_b^{-k}T_aT_c^{-1}T_b^kT_cT_a^{-1}T_e^k.
\end{align*}
Since $T_b^k\in \s$ and $T_a,T_c,T_e\in  \s'$, we have $\stab_{\lmap_k(S_2)}(e)=\langle \s \cup \{T_a,T_c,T_e,\iota\}\cup \s''' \rangle$, where $$\s'''=\{T_b^{-1}T_c^n(T_aT_b)^6T_c^{-n}T_b\mid n\in \Z\}.$$ Furthermore, since $[T_b,(T_aT_b)^6]=1$ we have
\begin{align*}
T_b^{-1}T_c^n(T_aT_b)^6T_c^{-n}T_b&=(T_b^{-1}T_c^nT_b)(T_aT_b)^6(T_b^{-1}T_c^{-n}T_b)\\
&=T_cT_b^nT_c^{-1}(T_aT_b)^6T_cT_b^{-n}T_c^{-1}.
\end{align*}
Since $T_c\in \s'$ and $T_b^k\in \s$, we have
\[
\stab_{\lmap_k(S_2)}(e)=\langle \s\cup \{T_a,T_c,T_e,\iota\}\cup \{T_b^jT_c^{-1}(T_aT_b)^6T_cT_b^{-j}\mid -k< j\leq 0\} \rangle.
\]
\end{proof}

\begin{cor}
\label{cor:stab_b_lmod}
We have $\stab_{\lmap_k(S_2)}(b)=\langle \{T_b^k,T_d,T_e,\iota,T_aT_c^{-1}T_b^kT_cT_a^{-1}\}\cup \s' \rangle$, where
\[
\s'=\{T_cT_dT_eT_d^{m+1}T_a^{-1}T_b^{-m-1}T_c^nT_d^{-1}(T_bT_c)^6(T_cT_dT_eT_d^{m+1}T_a^{-1}T_b^{-m-1}T_c^nT_d^{-1})^{-1}
\mid m,n\in \Z\}.
\]
\end{cor}

\begin{proof}
From short-exact sequence~(\ref{ses:stab_lmod}) and Lemmas~\ref{lem:stab_b_torelli}-\ref{lem:stab_b_matrices}, it follows that $$\stab_{\lmap_k(S_2)}(b)=\langle \{T_b^k,T_d,T_e,\iota\}\cup \s'\cup \s'' \rangle,$$
where
$$\s''=\{T_aT_bT_a^{-k}T_c^kT_b^{-1}T_a^{-1}T_e^k, T_dT_eT_dT_aT_bT_a^{-k}T_c^kT_b^{-1}T_a^{-1}T_e^k(T_dT_eT_d)^{-1}\}$$
and
$$\s'=\{T_cT_dT_eT_d^{m+1}T_a^{-1}T_b^{-m-1}T_c^nT_d^{-1}(T_bT_c)^6(T_cT_dT_eT_d^{m+1}T_a^{-1}T_b^{-m-1}T_c^nT_d^{-1})^{-1}
\mid m,n\in \Z\}.$$
Since $T_aT_bT_a^{-k}T_c^kT_b^{-1}T_a^{-1}T_e^{k}=T_b^{-k}T_aT_c^{-1}T_b^kT_cT_a^{-1}T_e^k$ and $T_b^k\in \stab_{\lmap_k(S_2)}(b)$, we have
$$\stab_{\lmap_k(S_2)}(b)=\langle \{T_b^k,T_d,T_e,\iota,T_aT_c^{-1}T_b^kT_cT_a^{-1}\} \cup \s'\rangle.$$
\end{proof}

For Steps 7-8 of Algorithm~\ref{algo}, we will use Theorem~\ref{thm:genset_gpaction_graph} to write a generating set for $\lmap_k(S_2)$.   

\begin{theorem}
\label{thm:genset_lmod}
For a prime number $k$ and the $k$-sheeted regular cyclic cover $p_k:S_{k+1}\to S_2$, we have $\lmap_k(S_2)=\langle \s\cup \{T_a,T_b^k,T_c,T_d,T_e,\iota\}\cup \s'\cup \s'' \rangle$, where $\s'=\{T_b^{1-j}T_aT_b^{1-\bar{j}}\mid 1\leq j<k\}$ and $\s''=\{(T_bT_c)^6,T_b^iT_c^jT_d(T_bT_c)^6T_d^{-1}T_c^{-j}T_b^{-i}\mid 1\leq i,j<k\}$.
\end{theorem}

\begin{proof}
We observe that the union of edges $\{\overline{e_1},\overline{e_2}\}$ and $\{\overline{e_2},\overline{e_3}\}$ forms a maximal tree in $\overline{\N_k(S_2)}$. In the complement of this maximal tree, there are only loops. For these loops in $\overline{\N_k(S_2)}$, we have already chosen edges in $\N(S_2)$ while implementing the Step 3 of Algorithm~\ref{algo}. They are as follows: for loops in $\overline{\N_k(S_2)}$, we chose edges $\{c,d\}$, $\{b,T_b^{-i_1}(a)\}$, and $\{b,T_b^{-i_2}(c)\}$ in $\N(S_2)$ corresponding to orbits $O_{\{e_3,e_4\}}$, $O_{\{e_2,e_1+i_1e_2\}}$, and $O_{\{e_2,e_1+i_2e_2+e_3\}}$, respectively, where $1\leq i_1,i_2<k$.

By Theorem~\ref{thm:genset_gpaction_graph}, for each loop $\{v,v\}\in E(\overline{\N_k(S_2)})$ at any vertex $v$ and a representative edge $\{v_1,v_2\}\in E(\N(S_2))$, we need an element $F\in \lmap_k(S_2)$ such that $F(v_1)=v_2$. For the edge $\{c,d\}$, we have $T_cT_d(c)=d$. For the edge $\{b,T_b^{-i_1}(a)\}$, we have
$$T_b^{1-i_1}T_aT_b^{1-\bar{i_1}}(b)=T_b^{-i_1}T_bT_a(b)=T_b^{-i_1}(a).$$ Similarly, for the edge $\{b,T_b^{-i_2}(c)\}$, we have
$$T_b^{1-i_2}T_aT_b^{1-\bar{i_2}}T_a^{-1}T_c(b)=T_b^{-i_2}T_a^{1-\bar{i_2}}T_bT_c(b)=T_b^{-i_2}(c).$$
By Proposition~\ref{prop:sym_mat_s2}, for $1\leq j<k$, we have $T_b^{1-j}T_aT_b^{1-\bar{j}}\in \lmap_k(S_2)$. Hence, by Theorem~\ref{thm:genset_gpaction_graph}, we get
\[
\lmap_k(S_2)=\langle \stab_{\lmap_k(S_2)}(a)\cup \stab_{\lmap_k(S_2)}(b)\cup \stab_{\lmap_k(S_2)}(c)\cup \s_1\rangle,
\]
where $\s_1=\{T_cT_d,T_b^{1-j}T_aT_b^{1-\bar{j}}\mid 1\leq j<k\}$. It follows from Lemma~\ref{lem:stab_a} and Corollaries~\ref{cor:stab_e_lmod}-\ref{cor:stab_b_lmod} that $\lmap_k(S_2)=\langle \s\cup \s_1\cup \s_2\cup \s_3\cup \s_1 \rangle$, where $\s_1=\{(T_aT_b)^6,T_b^iT_c^{-1}(T_aT_b)^6T_cT_b^{-i}\mid -k< i\leq -1\}$, $\s_2=\{T_a,T_b^k,T_c,T_d,T_e,\iota\}$, and $\s_3=\{(T_bT_c)^6,T_b^iT_c^nT_d^{-1}(T_bT_c)^6T_dT_c^{-n}T_b^{-i}\mid -k< i\leq -1, n\in \Z\}$. For $n\in \Z$ and $i+j=0$, we choose $T_b^iT_c^nT_d^{-1}(T_bT_c)^6T_dT_c^{-n}T_b^{-i}\in \s_3$ and $T_b^{1-j}T_aT_b^{1-\bar{j}}\in \s_1$. We have
\begin{align*}
&T_b^iT_c^nT_d^{-1}(T_bT_c)^6T_dT_c^{-n}T_b^{-i}\\
&=(T_b^{1-j}T_aT_b^{1-\bar{j}})(T_b^{1-j}T_aT_b^{1-\bar{j}})^{-1}T_b^iT_c^nT_d^{-1}(T_bT_c)^6T_dT_c^{-n}T_b^{-i}(T_b^{1-j}T_aT_b^{1-\bar{j}})(T_b^{1-j}T_aT_b^{1-\bar{j}})^{-1}\\
&=(T_b^{1-j}T_aT_b^{1-\bar{j}}T_a^{-1})(T_aT_b^{\bar{j}-1}T_a^{-1})T_b^{-1}T_c^nT_d^{-1}(T_bT_c)^6((T_b^{1-j}T_aT_b^{1-\bar{j}}T_a^{-1})(T_aT_b^{\bar{j}-1}T_a^{-1})T_b^{-1}T_c^nT_d^{-1})^{-1}\\
&=(T_b^{1-j}T_aT_b^{1-\bar{j}}T_a^{-1})(T_b^{-1}T_a^{\bar{j}-1}T_c^nT_d^{-1}(T_bT_c)^6((T_b^{1-j}T_aT_b^{1-\bar{j}}T_a^{-1})(T_b^{-1}T_a^{\bar{j}-1}T_c^nT_d^{-1})^{-1}\\
&=(T_b^{1-j}T_aT_b^{1-\bar{j}}T_a^{-1})(T_b^{-1}T_c^nT_b)T_b^{-1}T_a^{\bar{j}-1}T_d^{-1}(T_bT_c)^6((T_b^{1-j}T_aT_b^{1-\bar{j}}T_a^{-1})(T_b^{-1}T_c^nT_b)T_b^{-1}T_a^{\bar{j}-1}T_d^{-1})^{-1}\\
&=(T_b^{1-j}T_aT_b^{1-\bar{j}}T_a^{-1}T_c)T_b^nT_c^{-1}T_b^{-1}T_a^{\bar{j}-1}T_d^{-1}(T_bT_c)^6((T_b^{1-j}T_aT_b^{1-\bar{j}}T_a^{-1}T_c)T_b^nT_c^{-1}T_b^{-1}T_a^{\bar{j}-1}T_d^{-1})^{-1}.
\end{align*}
In the generating set for $\lmap_k(S_2)$ described before, by the above computations, we can replace the infinite set $\s_3$ with the finite set
\[
\s_3'=\{(T_bT_c)^6,T_b^iT_c^{-1}T_b^{-1}T_a^{j-1}T_d^{-1}(T_bT_c)^6(T_b^iT_c^{-1}T_b^{-1}T_a^{j-1}T_d^{-1})^{-1}\mid 1\leq j<k,-k<i\leq 0\}.
\]
We further simply elements of $\s_3'$. For $i\neq 0$ and $i+\ell=0$, we choose $T_b^{1-\ell}T_aT_b^{1-\bar{\ell}}\in \s_1$ and $T_b^iT_c^{-1}T_b^{-1}T_a^{j-1}T_d^{-1}(T_bT_c)^6(T_b^iT_c^{-1}T_b^{-1}T_a^{j-1}T_d^{-1})^{-1}\in \s_3'$ We have
\begin{align*}
&T_b^iT_c^{-1}T_b^{-1}T_a^{j-1}T_d^{-1}(T_bT_c)^6(T_b^iT_c^{-1}T_b^{-1}T_a^{j-1}T_d^{-1})^{-1}\\
&=(T_b^{1-\ell}T_aT_b^{1-\bar{\ell}})T_b^{\bar{\ell}-1}(T_a^{-1}T_b^{-1}T_c^{-1}T_d^{-1})T_b^{-1}T_a^{j-1}(T_bT_c)^6((T_b^{1-\ell}T_aT_b^{1-\bar{\ell}})T_b^{\bar{\ell}-1}(T_a^{-1}T_b^{-1}T_c^{-1}T_d^{-1})T_b^{-1}T_a^{j-1})^{-1}\\
&=(T_b^{1-\ell}T_aT_b^{1-\bar{\ell}})T_b^{\bar{\ell}-1}(T_c^{-1}T_b^{j-1}T_c)(T_cT_d)^6((T_b^{1-\ell}T_aT_b^{1-\bar{\ell}})T_b^{\bar{\ell}-1}(T_c^{-1}T_b^{j-1}T_c))^{-1}\\
&=(T_b^{1-\ell}T_aT_b^{1-\bar{\ell}})T_b^{\bar{\ell}}T_c^{j-1}T_b^{-1}(T_cT_d)^6((T_b^{1-\ell}T_aT_b^{1-\bar{\ell}})T_b^{\bar{\ell}}T_c^{j-1}T_b^{-1})^{-1}\\
&=(T_b^{1-\ell}T_aT_b^{1-\bar{\ell}})T_b^{\bar{\ell}}T_c^{j}T_d(T_d^{-1}T_c^{-1}T_b^{-1}(T_cT_d)^6((T_b^{1-\ell}T_aT_b^{1-\bar{\ell}})T_b^{\bar{\ell}}T_c^{j}T_d(T_d^{-1}T_c^{-1}T_b^{-1})^{-1}\\
&=(T_b^{1-\ell}T_aT_b^{1-\bar{\ell}})T_b^{\bar{\ell}}T_c^jT_d(T_bT_c)^6((T_b^{1-\ell}T_aT_b^{1-\bar{\ell}})T_b^{\bar{\ell}}T_c^jT_d)^{-1}.
\end{align*}
When $i=0$, we have
\begin{align*}
&T_c^{-1}T_b^{-1}T_a^{j-1}T_d^{-1}(T_bT_c)^6(T_c^{-1}T_b^{-1}T_a^{j-1}T_d^{-1})^{-1}\\
&=T_c^{-1}T_d^{-1}(T_b^{-1}T_a^{j-1}T_b)(T_bT_c)^6(T_c^{-1}T_d^{-1}(T_b^{-1}T_a^{j-1}T_b))^{-1}\\
&=T_c^{-1}T_d^{-1}T_a(T_b^{j-1}T_a^{-1}T_b^{\bar{j}-1})(T_bT_c)^6(T_c^{-1}T_d^{-1}T_a(T_b^{j-1}T_a^{-1}T_b^{\bar{j}-1}))^{-1}.
\end{align*}
Hence $\s_3'$ can be replaced with $\s''=\{(T_bT_c)^6,T_b^iT_c^jT_d(T_bT_c)^6T_d^{-1}T_c^{-j}T_b^{-i}\mid 1\leq i,j<k\}$. Finally, we reduce the set $\s_1$ to the set $\{(T_bT_c)^6\}$. For $i+j=0$, we choose $T_b^iT_c^{-1}(T_aT_b)^6T_cT_b^{-i}\in \s_1$ and $T_b^{1-j}T_aT_b^{1-\bar{j}}\in \s_1$. We have
\begin{align*}
&T_b^iT_c^{-1}(T_aT_b)^6T_cT_b^{-i}\\
&=(T_b^{1-j}T_aT_b^{1-\bar{j}})(T_b^{1-j}T_aT_b^{1-\bar{j}})^{-1}T_b^iT_c^{-1}(T_aT_b)^6T_cT_b^{-i}(T_b^{1-j}T_aT_b^{1-\bar{j}})(T_b^{1-j}T_aT_b^{1-\bar{j}})^{-1}\\
&=(T_b^{1-j}T_aT_b^{1-\bar{j}})T_b^{\bar{j}-1}T_a^{-1}T_b^{-1}T_c^{-1}(T_aT_b)^6((T_b^{1-j}T_aT_b^{1-\bar{j}})T_b^{\bar{j}-1}T_a^{-1}T_b^{-1}T_c^{-1})^{-1}\\
&=(T_b^{1-j}T_aT_b^{1-\bar{j}})T_b^{\bar{j}-1}(T_bT_c)^6((T_b^{1-j}T_aT_b^{1-\bar{j}})T_b^{\bar{j}-1})^{-1}\\
&=(T_b^{1-j}T_aT_b^{1-\bar{j}})(T_bT_c)^6(T_b^{1-j}T_aT_b^{1-\bar{j}})^{-1}
\end{align*}
Hence, $\lmap_k(S_2)=\langle \s\cup \{T_a,T_b^k,T_c,T_d,T_e,\iota\}\cup \s'\cup \s'' \rangle$, where $\s'=\{T_b^{1-j}T_aT_b^{1-\bar{j}}\mid 1\leq j<k\}$ and $\s''=\{(T_bT_c)^6,T_b^iT_c^jT_d(T_bT_c)^6T_d^{-1}T_c^{-j}T_b^{-i}\mid 1\leq i,j<k\}$.
\end{proof}

As a corollary of Theorem~\ref{thm:genset_lmod}, we recover the finite generating set for $\lmap_2(S_2)$ obtained in~\cite[Corollary 5.5]{dhanwani21} and obtain a simpler finite generating set for $\lmap_3(S_2)$.

\begin{cor}
\label{cor:genset_lmod2}
We have $\lmap_2(S_2)=\langle T_a,T_b^2,T_c,T_d,T_e\rangle$.
\end{cor}

\begin{proof}
We put $k=2$ in the generating set obtained in Theorem~\ref{thm:genset_lmod}. It can be computed that $\phi(\Gamma_0(2))=\langle \Psi(T_a), \Psi(T_b^2)\rangle$. Therefore, $\s''=\{(T_bT_c)^6,T_bT_cT_d(T_bT_c)^6T_d^{-1}T_c^{-1}T_b^{-1}\}$, $\s=\{T_a,T_b^2\}$, and $\s'=\{T_a\}$. We have
$$\iota=T_eT_dT_cT_bT_a^2T_bT_cT_dT_e=T_eT_d T_c(T_bT_a^2T_b^{-1})T_b^2T_cT_dT_e= T_eT_dT_c(T_a^{-1}T_b^2T_a)T_b^2T_cT_dT_e.$$
Furthermore,
$$(T_bT_c)^6=(T_b(T_cT_bT_c)T_bT_c)^2=  (T_b(T_bT_cT_b)T_bT_c)^2=(T_b^2T_c)^4$$
and $T_bT_cT_d(T_bT_c)^6T_d^{-1}T_c^{-1}T_b^{-1}=(T_cT_d)^6$. Hence, $\lmap_2(S_2)=\langle T_a,T_b^2,T_c,T_d,T_e\rangle$.
\end{proof}

\begin{cor}
\label{cor:genset_lmod3}
We have $\lmap_3(S_2)=\langle T_a,T_b^3,T_c,T_d,T_e,\iota\rangle$.
\end{cor}

\begin{proof}
We put $k=3$ in the generating set obtained in Theorem~\ref{thm:genset_lmod}. It can be computed that $\phi(\Gamma_0(3))=\langle \Psi(T_a), \Psi(T_b^3),\Psi((T_aT_b)^3)\rangle$. We have $\s=\{ T_a,T_b^3,(T_aT_b)^3\}$, $\s'=\{T_b^{-1}T_aT_b^{-1},T_a\}$, and $\s''=\{(T_bT_c)^6,T_b^iT_c^jT_d(T_bT_c)^6T_d^{-1}T_c^{-j}T_b^{-i}\mid 1\leq i,j<3\}$. Since $(T_aT_b)^3(T_dT_e)^{-3}=\iota$, we have $(T_aT_b)^3=\iota (T_dT_e)^3$. Since
\begin{align*}
(T_aT_b)^3
&=T_aT_bT_aT_bT_aT_b\\
&=T_aT_b^2T_aT_b^2\\
& =T_aT_b(T_bT_aT_b^{-1})T_b^3\\
&=T_aT_bT_a^{-1}T_bT_aT_b^3\\
&=(T_b^{-1}T_aT_b^{-1})T_b^3T_aT_b^3,
\end{align*}
we have $$T_b^{-1}T_aT_b^{-1}=(T_aT_b)^3T_b^{-3}T_a^{-1}T_b^{-3}=i(T_dT_e)^3T_b^{-3}T_a^{-1}T_b^{-3}.$$
Now, we show that $(T_bT_c)^6=(T_b^3T_c)^3$. Since $(T_bT_c)^6=(T_b^2T_c)^4$, further simplifying, we get
\begin{align*}
(T_bT_c)^6
&=T_b^2T_cT_b^2T_cT_b^2T_cT_b^2T_c\\
&=T_b^3(T_b^{-1}T_cT_b)T_bT_cT_b^2T_cT_b^2T_c\\
&=T_b^3T_cT_b(T_c^{-1}T_bT_c)T_b^2T_cT_b^2T_c\\ &=T_b^3T_cT_b^2(T_cT_bT_c)T_b^2T_c\\
&=T_b^3T_cT_b^3T_cT_b^3T_c\\
&=(T_b^3T_c)^3.
\end{align*}
Now, we consider the elements of $\s''$. We have
\begin{align*}
\s''=\{&T_bT_cT_d(T_bT_c)^6T_d^{-1}T_c^{-1}T_b^{-1},T_b^2T_cT_d(T_bT_c)^6T_d^{-1}T_c^{-1}T_b^{-2},\\
&T_b^2T_c^2T_d(T_bT_c)^6T_d^{-1}T_c^{-2}T_b^{-2},T_bT_c^2T_d(T_bT_c)^6T_d^{-1}T_c^{-2}T_b^{-1}\}.
\end{align*}
We have
\begin{align*}
&T_b^2T_c^2T_d(T_bT_c)^6T_d^{-1}T_c^{-2}T_b^{-2}\\
&=T_b^3(T_b^{-1}T_c^2T_b)T_d(T_bT_c)^6(T_b^3(T_b^{-1}T_c^2T_b)T_d)^{-1}\\
&=T_b^3T_cT_b^2T_c^{-1}T_d(T_bT_c)^6(T_b^3T_cT_b^2T_c^{-1}T_d)^{-1}\\
&=(T_b^3T_cT_b^3)T_b^{-1}(T_c^{-1}T_dT_c)(T_bT_c)^6((T_b^3T_cT_b^3)T_b^{-1}(T_c^{-1}T_dT_c))^{-1}\\
&=T_b^3T_cT_b^3T_dT_c(T_c^{-1}T_b^{-1}T_c)T_d^{-1}(T_bT_c)^6(T_b^3T_cT_b^3T_dT_c(T_c^{-1}T_b^{-1}T_c)T_d^{-1})^{-1}\\
&=T_b^3T_cT_b^3T_dT_cT_bT_c^{-1}T_d^{-1}(T_bT_c)^6(T_b^3T_cT_b^3T_dT_cT_bT_c^{-1}T_d^{-1})^{-1}\\
&=T_b^3T_cT_b^3T_dT_cT_b^3(T_b^{-2}T_c^{-1}T_d^{-1})(T_bT_c)^6(T_b^3T_cT_b^3T_dT_cT_b^3(T_b^{-2}T_c^{-1}T_d^{-1}))^{-1}\\
&=(T_b^3T_cT_b^3T_dT_cT_b^3)T_b^{-1}(T_cT_d)^6((T_b^3T_cT_b^3T_dT_cT_b^3)T_b^{-1})^{-1}\\
&=T_b^3T_cT_b^3T_dT_cT_b^3T_cT_d(T_d^{-1}T_c^{-1}T_b^{-1})(T_cT_d)^6(T_b^3T_cT_b^3T_dT_cT_b^3T_cT_d(T_d^{-1}T_c^{-1}T_b^{-1}))^{-1}\\
&=T_b^3T_cT_b^3T_dT_cT_b^3T_cT_d(T_bT_c)^6(T_b^3T_cT_b^3T_dT_cT_b^3T_cT_d)^{-1}\\
&=T_b^3T_cT_b^3T_dT_cT_b^3T_cT_d(T_b^3T_c)^3(T_b^3T_cT_b^3T_dT_cT_b^3T_cT_d)^{-1}.
\end{align*}
We also have $T_bT_cT_d(T_bT_c)^6T_d^{-1}T_c^{-1}T_b^{-1}=(T_cT_d)^6$. Furthermore,
\begin{align*}
T_b^2T_cT_d(T_bT_c)^6T_d^{-1}T_c^{-1}T_b^{-2}
&=T_b(T_cT_d)^6T_b^{-1}\\
&=T_c^{-1}T_d^{-1}(T_dT_cT_b(T_cT_d)^6T_b^{-1}T_c^{-1}T_d^{-1})T_dT_c\\
&=T_c^{-1}T_d^{-1}(T_bT_c)^6T_dT_c\\
&=T_c^{-1}T_d^{-1}(T_b^3T_c)^3T_dT_c.
\end{align*}
Finally, we have
\begin{align*}
T_bT_c^2T_d(T_bT_c)^6T_d^{-1}T_c^{-2}T_b^{-1}
&=(T_bT_c^2T_b^{-1})T_d(T_bT_c)^6((T_bT_c^2T_b^{-1})T_d)^{-1}\\
&=T_c^{-1}(T_b^2T_cT_d(T_bT_c)^6T_d^{-1}T_c^{-1}T_b^{-2})T_c\\
&=T_c^{-1}(T_c^{-1}T_d^{-1}(T_b^3T_c)^3T_dT_c)T_c\\
&=T_c^{-2}T_d^{-1}(T_b^3T_c)^3T_dT_c^2
\end{align*}
Hence, $\lmap_3(S_2)=\langle T_a,T_b^3,T_c,T_d,T_e,\iota\rangle$.
\end{proof}

Let $F_{g,k}\in \map(S_{g_k})$ be a periodic mapping class of order $k$ inducing the cover $p_k$. For $g\geq 3$, an explicit finite generating set for $N_{\map(S_{g_2})}(F_{g,2})$ was obtained in~\cite[Corollary 3.10]{dhanwani21} by lifting the finite generating set of $\lmap_2(S_g)$ under $\varphi:N_{\map(S_{g_2})}(F_{g,2})\to \lmap_2(S_g)$. By lifting the generating set described in Corollaries~\ref{cor:genset_lmod2}-\ref{cor:genset_lmod3}, we get the following result.

\begin{figure}[ht]
\centering
\includegraphics[scale=0.4]{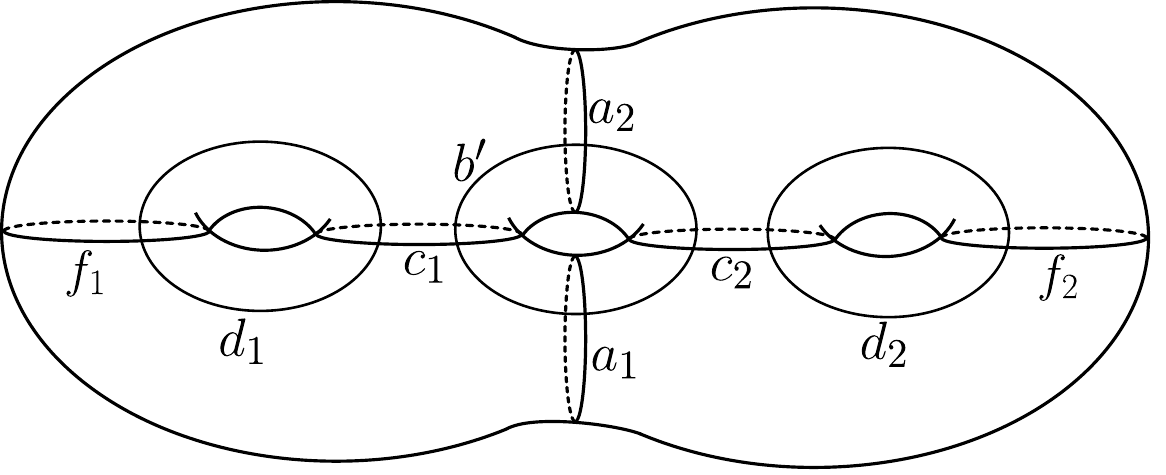}
\caption{The simple closed curves on $S_3$ used to describe a generating set for $N_{\map(S_3)}(F_{2,2})$.}
\label{fig:s3_curves}
\end{figure} 

\begin{figure}[ht]
\centering
\includegraphics[scale=1.2]{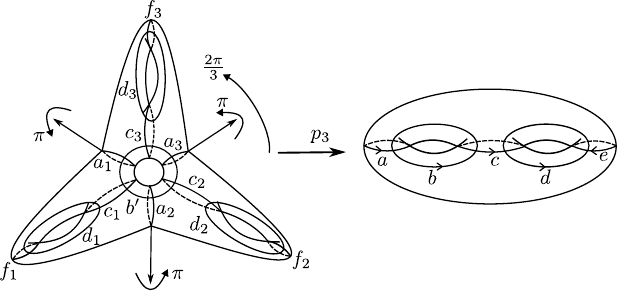}
\caption{A $3$-sheeted cover $S_4\to S_2$ induced by a free $2\pi/3$-rotation of $S_4$.}
\label{fig:cover_p3}
\end{figure}

\begin{cor}
\label{cor:genset_norm23}
We have
$$N_{\map(S_3)}(F_{2,2})=\langle F_{2,2}, T_{a_1}T_{a_2}, T_{b'}, T_{c_1}T_{c_2}, T_{d_1}T_{d_2},T_{f_1}T_{f_2}\rangle$$
and
$$N_{\map(S_4)}(F_{2,3})=\langle F_{2,3}, T_{a_1}T_{a_2}T_{a_3}, T_{b'}, T_{c_1}T_{c_2}T_{c_3}, T_{d_1}T_{d_2}T_{d_3},T_{f_1}T_{f_2}T_{f_3},\iota_1\iota_2\iota_3\rangle$$
for simple closed curves shown in Figure~\ref{fig:cover_p3}-\ref{fig:s3_curves} and $\iota_1,\iota_2,\iota_3\in \map(S_4)$ are involutions represented by $\pi$-rotation of $S_4$ as shown in Figure~\ref{fig:cover_p3}. 
\end{cor}

\subsection{A finite generating set for $\lmap_p(S_2)$ for a $\Z_2\oplus \Z_2$ cover $S_5\to S_2$}
For $\Z_2\oplus \Z_2=\langle s\mid s^2=1\rangle \oplus \langle t\mid t^2=1\rangle$, consider the map $H_1(S_2;\Z)\to \Z_2\oplus \Z_2$ given by $(e_1,e_2,e_3,e_4)\mapsto (s,t,0,0)$. This map corresponds to a $\Z_2\oplus \Z_2$-cover $p:S_5\to S_2$. In this subsection, we compute a finite generating set for $\lmap_p(S_2)$ using Algorithm~\ref{algo}.

For Step 1 of Algorithm~\ref{algo}, by Theorem~\ref{thm:sym_criteria}, we get the following symplectic criteria for $\lmap_p(S_2)$.
\begin{cor}
\label{cor:klein_symp}
For $\Psi(F)=(a_{ij})_{4\times 4}$, $F\in \lmap_p(S_2)$ if and only if $a_{13},a_{14},a_{23},a_{24},a_{31},a_{32}$,  $a_{41},a_{42}$ are even.
\end{cor}

For Step 2 of Algorithm~\ref{algo}, we compute the structure of the finite graph $\overline{\N(S_2)}\cong\overline{\N_2(S_2)}$.

\begin{lem}
\label{lem:quotient_graph_klein}
The graph $\overline{\N_2(S_2)}$ has three vertices corresponding to orbits $\overline{e_1}= O_{e_1}$, $\overline{e_3} =O_{e_3}$, and $\overline{e_1-e_3}= O_{e_1-e_3}$. Furthermore, there is one loop at each of $\overline{e_1}$ and $\overline{e_3}$, two loop at $\overline{e_1-e_3}$, one edge between $\overline{e_1}$ and $\overline{e_1-e_3}$, one edge between $\overline{e_3}$ and $\overline{e_1-e_3}$, and no other edges.   
\end{lem}

\begin{proof}
By Corollary~\ref{cor:klein_symp}, it follows that $O_{e_1}=\{e_1,e_2,e_1+e_2\}$, $O_{e_3}=\{e_3,e_4,e_3+e_4\}$, and $O_{e_1-e_3}=\{e_1+e_3,e_1+e_2+e_3,e_1+e_2+e_4,e_1+e_2+e_3+e_4, e_2+e_3+e_4,e_2+e_3,e_2+e_4,e_1+e_4, e_1+e_3+e_4 \}$. Hence, $\overline{\N_2(S_2)}$ has three vertices.

Furthermore, there is one loop at $\overline{e_1}$ and $\overline{e_3}$ corresponding to orbit $O_{\{e_1,e_2\}}$ and $O_{\{e_3,e_4\}}$, respectively. It can be shown that there are two loops at $\overline{e_1-e_3}$ corresponding to the orbits $O_{\{e_1+e_3,e_1+e_2+e_3\}}=\{\{e_1+e_3,e_1+e_2+e_3\},\{e_1+e_3,e_2+e_3\}\}$ and $O_{\{e_1+e_3,e_1+e_3+e_4\}}=\{\{e_1+e_3,e_1+e_3+e_4\},\{e_1+e_3,e_1+e_4\}\}$.

Since $\hat{i}_2(e_1,e_3)=0$, there are no edges between $\overline{e_1}$ and $\overline{e_3}$. It can be shown that $O_{\{e_2,e_1+e_3\}}=\{\{e_1,e_1+e_3\},\{e_2,e_1+e_3\},\{e_1+e_2,e_2+e_3\},\{e_1,e_1+e_2+e_3\},\{e_1,e_1+e_2+e_4\},\{e_1,e_1+e_2+e_3+e_4\},\{e_1,e_2+e_3+e_4\},\{e_1,e_2+e_3\},\{e_1,e_2+e_4\},\{e_2,e_1+e_2+e_3\},\{e_2,e_1+e_2+e_4\},\{e_2,e_1+e_2+e_3+e_4\},\{e_2,e_1+e_4\},\{e_2,e_1+e_3+e_4\},\{e_1+e_2,e_1+e_3\},\{e_1+e_2,e_2+e_3+e_4\},\{e_1+e_2,e_2+e_4\},\{e_1+e_2,e_1+e_4\},\{e_1+e_2,e_1+e_3+e_4\}\}$. Thus, there is exactly one edge between vertices $\overline{e_1}$ and $\overline{e_1-e_3}$. Similarly, it can be shown that there is exactly one edge between $\overline{e_1-e_3}$ and $\overline{e_3}$.
\end{proof}

For Step 3 of Algorithm~\ref{algo}, we choose representative curves $b$, $c$, and $d$ for vertices $\overline{e_1}$, $\overline{e_1+e_3}$, and $\overline{e_3}$ of $\overline{\N_2(S_2)}$, respectively. Since the complement of loops in $\overline{\N_2(S_2)}$ is the maximal tree, for loops of $\overline{\N_2(S_2)}$, we choose edges $\{a,b\}$, $\{e,d\}$, $\{c,T_b(c)\}$, and $\{c,T_d(c)\}$ corresponding to edge orbits $O_{\{e_1,e_2\}}$, $O_{\{e_3,e_4\}}$, $O_{\{e_1+e_3,e_1+e_2+e_3\}}$, and $O_{\{e_1+e_3,e_1+e_3+e_4\}}$, respectively. Since $T_aT_b(a)=b$, $T_eT_d(e)=d$, and $T_aT_b,T_eT_d\in \lmap_p(S_2)$, we have $\stab_{\lmap_p(S_2)}(b)=T_aT_b\stab_{\lmap_p(S_2)}(a)T_b^{-1}T_a^{-1}$ and $\stab_{\lmap_p(S_2)}(d)=T_eT_d\stab_{\lmap_p(S_2)}(e)T_d^{-1}T_e^{-1}$. For easier computations, we will derive generating sets for $\stab_{\lmap_p(S_2)}(a)$ and $\stab_{\lmap_p(S_2)}(e)$ instead of $\stab_{\lmap_p(S_2)}(b)$ and $\stab_{\lmap_p(S_2)}(d)$, respectively.

For Step 4 of Algorithm~\ref{algo}, we compute $\stab_{\I(S_2)}(\gamma)$ for $\gamma=a,c,e$. By Lemma~\ref{lem:stab_e_torelli}, we recall that
\[
\stab_{\I(S_2)}(e)=\langle
T_a^{-m-1}T_c^mT_b^{-1}T_c^n(T_aT_b)^6T_c^{-n}T_bT_c^{-m}T_a^{m+1}\mid m,n\in \Z\rangle.
\]

\begin{lem}
\label{lem:stab_torelli_klien_a}
We have
\[
\stab_{\I(S_2)}(a)=\langle T_e^{-m-1}T_c^mT_d^{-1}T_c^n(T_eT_d)^6T_c^{-n}T_dT_c^{-m}T_e^{m+1}\mid m,n\in \Z  \rangle.
\]
\end{lem}

\begin{proof}
Consider the involution $F\in \map(S_2)$ (a handle swap map) such that $F(e)=a$. Since $\stab_{\I(S_2)}(a)=F\stab_{\I(S_2)}(e)F^{-1}$, $F(b)=d$, and $F(c)=c$, we have
\[
\stab_{\I(S_2)}(a)=\langle T_e^{-m-1}T_c^mT_d^{-1}T_c^n(T_eT_d)^6T_c^{-n}T_dT_c^{-m}T_e^{m+1}\mid m,n\in \Z  \rangle.
\]
\end{proof}

\begin{lem}
\label{lem:stab_torelli_klien_c}
We have
\[
\stab_{\I(S_2)}(c)=\langle (T_a^{-1}T_e)^{m+1}T_dT_eT_b^{-1}T_c^nT_bT_d^{-1}T_c(T_aT_b)^6T_c^{-1}T_dT_b^{-1}T_c^{-n}T_bT_e^{-1}T_d^{-1}(T_aT_e^{-1})^{m+1}\mid m,n\in \Z\rangle
\]
\end{lem}

\begin{proof}
For the computation of $\stab_{\I(S_2)}(c)$, since for $G=T_dT_cT_eT_d$, we have $G(e)=c$. Thus, $\stab_{\I(S_2)}(c)=G\stab_{\I(S_2)}(e)G^{-1}$. Furthermore, we observe that $G(a)=a$, $G(b)=T_dT_c(b)=:x$, and $G(c)=e$. Therefore,
\[
\stab_{\I(S_2)}(c)=\langle
T_a^{-m-1}T_e^mT_x^{-1}T_e^n(T_aT_x)^6T_e^{-n}T_xT_e^{-m}T_a^{m+1}\mid m,n\in \Z\rangle.
\]
Simplifying, we get
\begin{align*}
&T_a^{-m-1}T_e^mT_x^{-1}T_e^n(T_aT_x)^6T_e^{-n}T_xT_e^{-m}T_a^{m+1}\\
&=T_a^{-m-1}T_e^mT_dT_cT_b^{-1}T_c^{-1}(T_d^{-1}T_e^nT_d)T_c(T_aT_b)^6(T_a^{-m-1}T_e^mT_dT_cT_b^{-1}T_c^{-1}(T_d^{-1}T_e^nT_d)T_c)^{-1}\\
&=T_a^{-m-1}T_e^mT_dT_cT_b^{-1}T_c^{-1}T_eT_d^nT_e^{-1}T_c(T_aT_b)^6(T_a^{-m-1}T_e^mT_dT_cT_b^{-1}T_c^{-1}T_eT_d^nT_e^{-1}T_c)^{-1}\\
&=T_a^{-m-1}T_e^mT_dT_eT_cT_b^{-1}(T_c^{-1}T_d^nT_c)(T_aT_b)^6(T_a^{-m-1}T_e^mT_dT_eT_cT_b^{-1}(T_c^{-1}T_d^nT_c))^{-1}\\
&=T_a^{-m-1}T_e^mT_dT_eT_cT_b^{-1}T_dT_c^nT_d^{-1}(T_aT_b)^6(T_a^{-m-1}T_e^mT_dT_eT_cT_b^{-1}T_dT_c^nT_d^{-1})^{-1}\\
&=T_a^{-m-1}T_e^mT_dT_eT_cT_d(T_b^{-1}T_c^nT_b)(T_aT_b)^6(T_a^{-m-1}T_e^mT_dT_eT_cT_d(T_b^{-1}T_c^nT_b))^{-1}\\
&=T_a^{-m-1}T_e^mT_dT_e(T_cT_dT_c)T_b^nT_c^{-1}(T_aT_b)^6(T_a^{-m-1}T_e^mT_dT_e(T_cT_dT_c)T_b^nT_c^{-1})^{-1}\\
&=T_a^{-m-1}T_e^m(T_dT_eT_d)T_cT_dT_b^nT_c^{-1}(T_aT_b)^6(T_a^{-m-1}T_e^m(T_dT_eT_d)T_cT_dT_b^nT_c^{-1})^{-1}\\
&=(T_a^{-1}T_e)^{m+1}T_dT_eT_cT_b^nT_dT_c^{-1}(T_aT_b)^6((T_a^{-1}T_e)^{m+1}T_dT_eT_cT_b^nT_dT_c^{-1})^{-1}\\
&=(T_a^{-1}T_e)^{m+1}T_dT_e(T_cT_b^nT_c^{-1})(T_cT_dT_c^{-1})(T_aT_b)^6((T_a^{-1}T_e)^{m+1}T_dT_e(T_cT_b^nT_c^{-1})(T_cT_dT_c^{-1}))^{-1}\\
&=(T_a^{-1}T_e)^{m+1}T_dT_eT_b^{-1}T_c^nT_bT_d^{-1}T_c(T_aT_b)^6((T_a^{-1}T_e)^{m+1}T_dT_eT_b^{-1}T_c^nT_bT_d^{-1}T_c)^{-1}
\end{align*}
Hence, the result follows.
\end{proof}

For Step 5 of Algorithm~\ref{algo}, we derive a finite generating set for $\stab_{\Psi(\lmap_p(S_2))}(\pm [\gamma])$ for $\gamma=a,c,e$. Since $\iota\in \lmap_p(S_2)$ and $\Psi(\iota)[\gamma]=-\I_{4\times 4}([\gamma
])=-[\gamma]$, we compute $\stab_{\Psi(\lmap_p(S_2))}([\gamma])$. For our computation we will use the following matrices
\[
M=\begin{pmatrix}
 1 & 0 & 0 & 0 \\ 
 0 & 1 & 0 & 1 \\ 
 -1 & 0 & 1 & 0 \\ 
 0 & 0 & 0 & 1
 \end{pmatrix},
M'=\begin{pmatrix}
1 & 0 & 0 & 1 \\ 
0 & 1 & 0 & 0 \\ 
0 & 1 & 1 & 0 \\ 
0 & 0 & 0 & 1
\end{pmatrix}
\text{, and }
N'=\begin{pmatrix}
1 & 0 & 1 & 0 \\ 
0 & 1 & 0 & 0 \\ 
0 & 0 & 1 & 0 \\ 
0 & -1 & 0 & 1
\end{pmatrix},
\]
where $M'=\Psi(T_e)\Psi(T_a)\Psi(T_c)^{-1}$ and $N'=\Psi(T_e)\Psi(T_d)\Psi(T_e)M'(\Psi(T_e)\Psi(T_d)\Psi(T_e))^{-1}$. 

\begin{lem}
\label{lem:klein_stab_matrix_a}
We have $\stab_{\Psi(\lmap_p(S_2))}(\pm [a])=\langle (M')^2,(N')^2,\Psi(T_a),\Psi(T_d),\Psi(T_e),\iota
\rangle$.
\end{lem}

\begin{proof}
From the description of $\Psi(\lmap_p(S_2))$, it can be seen that if $A\in \stab_{\Psi(\lmap_p(S_2))}([a]=e_1)$, then
\[
A=\begin{pmatrix}
1 & a_{12} & a_{13} & a_{14} \\ 
0 & 1 & 0 & 0 \\ 
0 & a_{32} & a_{33} & a_{34} \\ 
0 & a_{42} & a_{43} & a_{44}
\end{pmatrix},
\]
where $a_{ij}\in \Z$, $a_{32},a_{42},a_{14},a_{13}\in 2\Z$, and $a_{33}a_{44}-a_{43}a_{34}=1$. Pre-multiplying $\Psi(T_a)^{-a_{12}}A$ with a suitable product of powers of matrices $\Psi(T_d)$ and $\Psi(T_e)$, we obtain
\[
A'=\begin{pmatrix}
1 & 0 & a_{13} & a_{14} \\ 
0 & 1 & 0 & 0 \\ 
0 & a_{14} & 1 & 0 \\ 
0 & -a_{13} & 0 & 1
\end{pmatrix}.
\]
Since $\Psi(T_a)^{-a_{13}a_{14}}(M')^{-a_{14}}(N')^{-a_{13}}A'=I_{4\times 4}$, it follows that
\[
\stab_{\Psi(\lmap_p(S_2))}(\pm [a])=\langle (M')^2,(N')^2,\Psi(T_a),\Psi(T_d),\Psi(T_e),\iota
\rangle.
\]
\end{proof}

\begin{lem}
\label{lem:klein_stab_matrix_e}
We have $\stab_{\Psi(\lmap_p(S_2))}(\pm [e])=\langle M^2,(M')^2,\Psi(T_a),\Psi(T_b),\Psi(T_e),\iota
\rangle$.
\end{lem}

\begin{proof}
As before, it follows that if $A\in \stab_{\Psi(\lmap_p(S_2))}([e]=e_3)$, then
\[
A=\begin{pmatrix}
a_{11} & a_{12} & 0 & a_{14} \\ 
a_{21} & a_{22} & 0 & a_{24} \\ 
a_{31} & a_{32} & 1 & a_{34} \\ 
0 & 0 & 0 & 1
\end{pmatrix},
\]
where $a_{ij}\in \Z$, $a_{31},a_{32},a_{14},a_{24}\in 2\Z$, and $a_{11}a_{22}-a_{12}a_{21}=1$. Pre-multiplying $\Psi(T_e)^{-a_{34}}A$ with a suitable product of powers of $\Psi(T_a)$ and $\Psi(T_b)$, we obtain
\[
A'=\begin{pmatrix}
1 & 0 & 0 & a_{32} \\ 
0 & 1 & 0 & -a_{31} \\ 
a_{31} & a_{32} & 1 & 0 \\ 
0 & 0 & 0 & 1
\end{pmatrix}.
\]
Using matrices $M$ from Equation~\ref{eqn:M_matrices}, we have $\Psi(T_e)^{-a_{31}a_{32}}(M')^{-a_{32}}M^{a_{31}}A'=I_{4\times 4}$. Now, it follows that
\[
\stab_{\Psi(\lmap_p(S_2))}(\pm [e])=\langle M^2,(M')^2,\Psi(T_a),\Psi(T_b),\Psi(T_e),\iota
\rangle.
\]
\end{proof}

For the next computation, we will use the following matrices:
\[
J=\begin{pmatrix}
3 & 0 & 2 & 0 \\
0 & -1 & 0 & 2 \\ 
-2 & 0 & -1 & 0 \\
0 & -2 & 0 & 3 
\end{pmatrix},
P=\begin{pmatrix}
0 & -1 & 0 & 0 \\
1 & 0 & 0 & 1 \\ 
0 & 1 & 1 & 0 \\
0 & 0 & 0 & 1 
\end{pmatrix},\text{ and }
I=\begin{pmatrix}
1 & 0 & 0 & 0 \\
0 & 1 & 0 & 0 \\ 
0 & 0 & -1 & 0 \\
0 & 0 & 0 & -1 
\end{pmatrix}.
\] 

\begin{lem}
\label{lem:klein_stab_matrix_c}
We have
\[
\stab_{\Psi(\lmap_p(S_2))}(\pm[c])=\langle \Psi(T_a), \Psi(T_e), \Psi(T_c^2),\iota, M^2(N')^2, M^2(T_dT_e)^3, \Psi(T_dT_cT_b^2T_c^{-1}T_d^{-1})\rangle.
\]
\end{lem}

\begin{proof}
If $A\in\stab_{\Psi(\lmap_p(S_2))}([c]=e_3-e_1)$, then
\[
A=\begin{pmatrix}
2a_{13}+1 & a_{12} & 2a_{13} & 2a_{14} \\ 
2a_{23} & 2a_{42}+1 & 2a_{23} & 2a_{24} \\ 
2a_{31} & 2a_{32} & 2a_{31}+1 & a_{34} \\ 
2a_{23} & 2a_{42} & 2a_{23} & 2a_{24}+1
\end{pmatrix}.
\]
Since $\Psi(T_a),\Psi(T_e)\in \stab_{\Psi(\lmap_p(S_2))}([c])$, without loss of generality, we can assume that
\[
A=\begin{pmatrix}
2a_{13}+1 & 2a_{12} & 2a_{13} & 2a_{14} \\ 
2a_{23} & 2a_{42}+1 & 2a_{23} & 2a_{24} \\ 
2a_{31} & 2a_{32} & 2a_{31}+1 & 2a_{34} \\ 
2a_{23} & 2a_{42} & 2a_{23} & 2a_{24}+1
\end{pmatrix}.
\]

Then $A_1=\Psi(T_e^{(2a_{12}+2a_{32})}T_a^{(-2a_{12}-2a_{32})}T_c^{(-4a_{12}a_{42}-4a_{32}a_{42}+2a_{32})})J^{a_{42}}A$, has the following form 
\[
A_1=\begin{pmatrix}
1 & 0 & 0 & 0 \\ 
2a_{23} & 1 & 2a_{23} & 2a_{42}+2a_{24} \\ 
2a_{13}+2a_{31} & 0 & 2a_{13}+2a_{31}+1 & 2a_{12}+2a_{32}+2a_{14}+2a_{34} \\ 
2a_{23} & 0 & 2a_{23} & 2a_{42}+2a_{24}+1
\end{pmatrix}.
\]
By taking a conjugate with matrix $P$, $A_2=P^{-1}A_1P$ has the form
\[
A_2=\begin{pmatrix}
1 & 0 & 0 & 0 \\ 
0 & 1 & 0 & 0 \\ 
0 & 0 & 2a_{13}+2a_{31}+1 & 2a_{12}+2a_{32}+2a_{14}+2a_{34}) \\ 
0 & 0 & 2a_{23} & 2a_{22}+2a_{24}+1
\end{pmatrix}.
\]
Now, it can be seen that
\[
\begin{pmatrix}
 2a_{13}+2a_{31}+1 & 2a_{12}+2a_{32}+2a_{14}+2a_{34} \\ 
 2a_{23} & 2a_{22}+2a_{24}+1
\end{pmatrix}\in \Gamma(2),
\]
where $\Gamma(2)$ is the principle congruence subgroup of $\Sl(2,\Z)$ of level $2$. Since $A_2\in \langle \Psi(T_d)^2, \Psi(T_e)^2,I\rangle$, we have 
\[
\stab_{\Psi(\lmap_p(S_2))}([c])=\langle \Psi(T_a), \Psi(T_e), \Psi(T_c^2), P\Psi(T_d)^2P^{-1}, P\Psi(T_e)^2P^{-1},PIP^{-1} J\rangle.
\]
We have $P\Psi(T_e)^2P^{-1}=\Psi(T_e)^2$,  $P\Psi(T_d)^2P^{-1}=\Psi(T_{e}T_{a}^{-1}T_{d}T_{c}T_{b}^2(T_{e}T_{a}^{-1}T_{d}T_{c}))$, $PIP^{-1}=M^2\psi(T_dT_e)^3$, and $J=\iota (N')^{-2}M^{-2}$. Hence,
\[
\stab_{\Psi(\lmap_p(S_2))}(\pm[c])=\langle \Psi(T_a), \Psi(T_e), \Psi(T_c^2),\iota, M^2(N')^2, M^2(T_dT_e)^3, \Psi(T_dT_cT_b^2T_c^{-1}T_d^{-1})\rangle.
\]
\end{proof}

For Step 6 of Algorithm~\ref{algo}, we write $\stab_{\lmap_p(S_2)}(\gamma)$ for $\gamma=a,c,e$.

\begin{cor}
\label{cor:stab_lamod_klein_a}
We have $\stab_{\lmap_p(S_2)}(a)=\langle T_a,T_c^2,T_d,T_e,\iota,T_c^{-1}(T_eT_d)^6T_c\rangle$
\end{cor}

\begin{proof}
By Lemmaa~\ref{lem:stab_torelli_klien_a} and \ref{lem:klein_stab_matrix_a}, it follows that
\begin{align*}
\stab_{\lmap_p(S_2)}(a)
&=\langle \{T_a,T_c^2,T_d,T_e,\iota\}\cup \{T_c^iT_d^{-1}T_c^n(T_eT_d)^6T_c^{-n}T_dT_c^{-i}\mid 0\leq i\leq 1,n\in \Z\}\rangle\\
&=\langle \{T_a,T_c^2,T_d,T_e,\iota\}\cup \{T_c^{i+1}T_d^nT_c^{-1}(T_eT_d)^6T_cT_d^{-n}T_c^{-i-1}\mid 0\leq i\leq 1,n\in \Z\}\rangle\\
&=\langle T_a,T_c^2,T_d,T_e,\iota,T_c^{-1}(T_eT_d)^6T_c\rangle.
\end{align*}
\end{proof}

\begin{cor}
\label{cor:stab_lamod_klein_e}
We have $\stab_{\lmap_p(S_2)}(e)=\langle T_a,T_b,T_c^2,T_e,\iota,T_c^{-1}(T_aT_b)^6T_c\rangle$.
\end{cor}

\begin{proof}
By Lemmas~\ref{lem:stab_e_torelli} and~\ref{lem:klein_stab_matrix_e}, it follows that
\begin{align*}
\stab_{\lmap_p(S_2)}(e)
&=\langle \{T_a,T_b,T_c^2,T_e,\iota\}\cup \{T_c^iT_b^{-1}T_c^n(T_aT_b)^6T_c^{-n}T_bT_c^{-i}\mid 0\leq i\leq 1,n\in \Z\}\rangle\\
&=\langle \{T_a,T_b,T_c^2,T_e,\iota\}\cup \{T_c^{i+1}T_b^nT_c^{-1}(T_aT_b)^6T_cT_b^{-n}T_c^{-i-1}\mid 0\leq i\leq 1,n\in \Z\}\rangle\\
&=\langle T_a,T_b,T_c^2,T_e,\iota,T_c^{-1}(T_aT_b)^6T_c\rangle.
\end{align*}
\end{proof}

\begin{cor}
\label{cor:stab_lamod_klein_c}
We have
\begin{align*}
\stab_{\lmap_p(S_2)}(c)=
&\langle\{T_a,T_e,T_c^2,\iota\}\cup \{T_dT_eT_b^{-1}T_c^nT_bT_d^{-1}T_c(T_aT_b)^6T_c^{-1}T_dT_b^{-1}T_c^{-n}T_bT_e^{-1}T_d^{-1}\mid n\in \Z\}\\
&\cup \{ T_dT_e^2T_c^{-2}T_d^{-1}, T_bT_c^2(T_dT_e)^3, T_dT_b^{-1}T_c^2T_bT_d^{-1},T_bT_a^{-2}T_c^2T_b^{-1}T_dT_c^{-2}T_e^2T_d^{-1}\} \rangle.
\end{align*}
\end{cor}

\begin{proof}
It follows from Lemmas~\ref{lem:stab_torelli_klien_c} and \ref{lem:klein_stab_matrix_c} that
\begin{align*}
\stab_{\lmap_p(S_2)}(c)=\langle 
&\{T_a,T_e,T_c^2,\iota,T_dT_cT_b^2T_c^{-1}T_d^{-1},T_dT_e^2T_a^2T_c^{-2}T_d^{-1},T_aT_bT_a^{-2}T_c^2T_b^{-1}T_a^{-1}T_e^{-2}(T_dT_e)^3,\\
&T_aT_bT_a^{-2}T_c^2T_b^{-1}T_a^{-1}T_e^{-1}T_dT_e^2T_a^2T_c^{-2}T_d^{-1}T_e^{-1}\} \cup\\
&\{T_dT_eT_b^{-1}T_c^nT_bT_d^{-1}T_c(T_aT_b)^6T_c^{-1}T_dT_b^{-1}T_c^{-n}T_bT_e^{-1}T_d^{-1}\mid n\in \Z\} \rangle.
\end{align*}
We have $T_dT_e^2T_a^2T_c^{-2}T_d^{-1}=T_dT_e^2T_c^{-2}T_d^{-1}T_a^2$, $T_aT_bT_a^{-2}T_c^2T_b^{-1}T_a^{-1}T_e^{-2}(T_dT_e)^3=T_aT_bT_c^2(T_dT_e)^3$, $T_dT_cT_b^2T_c^{-1}T_d^{-1}=T_dT_b^{-1}T_c^2T_bT_d^{-1}$, and
$$T_aT_bT_a^{-2}T_c^2T_b^{-1}T_a^{-1}T_e^{-1}T_dT_e^2T_a^2T_c^{-2}T_d^{-1}T_e^{-1}=T_aT_e^{-1}T_bT_a^{-2}T_c^2T_b^{-1}T_dT_c^{-2}T_e^2T_d^{-1}T_aT_e^{-1}.$$
Hence,
\begin{align*}
\stab_{\lmap_p(S_2)}(c)=
&\langle\{T_a,T_e,T_c^2,\iota\}\cup \{T_dT_eT_b^{-1}T_c^nT_bT_d^{-1}T_c(T_aT_b)^6T_c^{-1}T_dT_b^{-1}T_c^{-n}T_bT_e^{-1}T_d^{-1}\mid n\in \Z\}\\
&\cup \{ T_dT_e^2T_c^{-2}T_d^{-1}, T_bT_c^2(T_dT_e)^3, T_dT_b^{-1}T_c^2T_bT_d^{-1},T_bT_a^{-2}T_c^2T_b^{-1}T_dT_c^{-2}T_e^2T_d^{-1}\} \rangle.
\end{align*}
\end{proof}

\begin{theorem}
\label{thm:genset_klein_lmod}
For the cover $p$ described as above, we have $\lmap_p(S_2)=\langle T_a,T_b,T_c^2,T_d,T_e\rangle$.
\end{theorem}

\begin{proof}
For representative edges of loops of $\overline{\N_2(S_2)}$, we choose elements of $\lmap_p(S_2)$ taking one vertex of edge to the adjacent one. For the edge $\{a,b\}$ and $T_aT_b\in \lmap_p(S_2)$, we have $T_aT_b(a)=b$. For the edge $\{d,e\}$ and $T_dT_e\in \lmap_p(S_2)$, we have $T_dT_e(d)=e$. For edges $\{c,T_b(c)\}$ and $\{c,T_d(c)\}$, $T_b,T_d\in \lmap_p(S_2)$ does the job. Hence, by Theorem~\ref{thm:genset_gpaction_graph}, we have
\begin{align*}
\lmap_p(S_2)=\langle \stab_{\lmap_p(S_2)}(b)\cup\stab_{\lmap_p(S_2)}(c)\cup\stab_{\lmap_p(S_2)}(d)\cup \{T_aT_b,T_dT_e,T_b,T_d\} \rangle.
\end{align*}
Since $(T_aT_b)^6=(T_dT_e)^6$, by Corollaries~\ref{cor:stab_lamod_klein_a}-\ref{cor:stab_lamod_klein_c}, we have
\begin{align*}
\lmap_p(S_2)=\langle \{T_a,T_b,T_c^2,T_d,T_e,\iota, T_c^{-1}(T_aT_b)^6T_c,T_c^{-1}T_bT_d^{-1}T_c(T_aT_b)^6T_c^{-1}T_dT_b^{-1}T_c\} \rangle.
\end{align*}
We have
\begin{align*}
&T_c^{-1}T_bT_d^{-1}T_c(T_aT_b)^6(T_c^{-1}T_bT_d^{-1}T_c)^{-1}\\
&=(T_bT_a)T_a^{-1}T_b^{-1}T_c^{-1}T_d^{-1}T_bT_c(T_aT_b)^6((T_bT_a)T_a^{-1}T_b^{-1}T_c^{-1}T_d^{-1}T_bT_c)^{-1}\\
&=T_bT_aT_cT_d(T_bT_c)^6(T_bT_aT_cT_d)^{-1}\\
&=(T_bT_aT_b^{-1})T_bT_cT_d(T_bT_c)^6 ((T_bT_aT_b^{-1})T_bT_cT_d)^{-1}\\
&=(T_bT_aT_b^{-1})(T_cT_d)^6((T_bT_aT_b^{-1}))^{-1}\\
&=(T_bT_aT_b^{-1})(T_c^2T_d)^4((T_bT_aT_b^{-1}))^{-1}.
\end{align*} 
Furthermore,
\begin{align*}
T_c^{-1}(T_aT_b)^6T_c=T_c^{-1}(T_aT_b^2)^4T_c=(T_aT_c^{-1}T_b^2T_c)^4=(T_aT_bT_c^2T_b^{-1})^4.
\end{align*}
Finally, as seen before, we have $\iota=(T_aT_b)^3(T_dT_e)^{-3}$. Hence, $\lmap_p(S_2)=\langle T_a,T_b,T_c^2,T_d,T_e\rangle$.
\end{proof}

\section*{Acknowledgments}
The authors would like to thank Prof.\,Martin Bridson for providing some of the key ideas used in Subsection 2.1. The first author is supported by the NBHM Postdoctoral Fellowship via grant number 0204/1/2023/R \&\,D-II/1792, the second author was partly supported by the Prime Minister Research Fellowship (PMRF) scheme instituted by the Ministry of Education, India. The third author is supported by the SERB MATRICS grant of the Government of India. 
\bibliographystyle{abbrv}
\bibliography{liftable_s2}
\end{document}